\documentclass[final,1p,times,sort&compress]{elsarticle}%
\usepackage{amsmath}
\usepackage{amsfonts}
\usepackage{amssymb}
\usepackage[colorlinks]{hyperref}
\usepackage{subfigure}
\newtheorem{mytheo}{Theorem}[section]

\newtheorem{lem}[mytheo]{Lemma}

\newtheorem{algo}[mytheo]{Algorithm}

\newtheorem{rem}{Remark}[section]
\newcommand{\norm}[1]{\left\Vert#1\right\Vert}
\newcommand{\abs}[1]{\left\vert#1\right\vert}

\newcommand{\eps}{\varepsilon}

\newcounter{problem}
\newenvironment{problem}{\refstepcounter{problem}\vspace{1.5ex}
	{\noindent\bf Problem
		\theproblem.}\hspace{0.3em}\parindent=0pt}{\vspace{1ex}}

\makeatletter
\def\@upcite#1#2{\textsuperscript{[{#1\if@tempswa , #2\fi}]}}
\makeatother
\newenvironment{proof}{\vspace{1ex}
	{\it Proof. }\hspace{0.3em}}{\vspace{1ex}} \journal{ }
\DeclareMathOperator{\sinc}{sinc}

\begin{document}
	\begin{frontmatter}
		\title{A novel class of linearly implicit energy-preserving schemes\\ for conservative systems}

		\author[1]{Xicui Li}
		\author[1]{Bin Wang\corref{cor1}}
		\author[2]{Xin Zou}
		
		\address[1]{ School of Mathematics and Statistics, Xi'an Jiaotong University, 710049 Xi'an, China}
		\address[2]{ China Huaneng Clean Energy Research Institute, 102209 Beijing, China}		
		
		\ead{lixicui@stu.xjtu.edu.cn, wangbinmaths@xjtu.edu.cn, x_zou@qny.chng.com.cn}\cortext[cor1]{Corresponding author.}

		
		\begin{abstract}
			We consider {a kind of} differential equations $\dot{y}(t)=R(y(t))y(t)+f(y(t))$ with energy conservation. Such conservative models appear for
			instance in quantum physics, engineering {and} molecular dynamics. A new class of energy-preserving schemes {is} constructed by the ideas of scalar auxiliary variable (SAV) and splitting,  from which the  nonlinearly implicit schemes have been improved to be linearly implicit. The energy conservation and error estimates are rigorously derived. Based on these results, it is shown that the new proposed schemes have  unconditionally energy stability and can be implemented with a cost of solving a linearly implicit system. Numerical experiments are done to confirm these good features of the new schemes.
			
		\end{abstract}
		
	\end{frontmatter}
	 \vskip0.5cm \pagestyle{myheadings}
	\thispagestyle{plain}
	\section{Introduction}
	This paper is devoted to  numerical solutions of the following differential equation:
	\begin{equation}\label{IVP}
		\dot{y}(t)=R(y(t))y(t)+f(y(t)),\quad y(0)=y_0, \quad t\in[0,T],
	\end{equation}
	where  the unknown function $y(t):[0,T]\to X$ is in a finite dimensional space $X$,
	$R(y)$ is an operator of $X$ and may depend on $y$, and $y_0$ is the given initial data. Such kind of system arises from
	a large amount of applications, such as mechanics, quantum physics, molecular dynamics, circuit simulations and engineering.
	In recent years, many researchers have considered two important conservative systems of \eqref{IVP}: oscillatory second-order
	differential equations and charged-particle dynamics, and both of them can be written in the form  \eqref{IVP}. A brief introduction on these two systems is given in the following {two} paragraphs.
	
	\textbf{Oscillatory second-order
		differential equations  (OSDE).}
	We are firstly concerned with the oscillatory second-order differential equations {(OSDE)} of the form
	\begin{equation}\label{oso}
		\ddot{q}(t)+\frac{1}{\eps^2}Aq(t)=F(q(t)),\quad q(0)=q_0, \quad \dot{q}(0):=p_0,\quad t\in[0,T],
	\end{equation}
	where $q,p:=\dot{q}: [0,T] \rightarrow \mathbb{R}^d$  represent generalized
	position and velocity respectively, $A$ is a positive {semi-definite} matrix of bounded norm independent of  $\eps$, $0<\eps \leq1$ is inversely proportional to the spectral
	radius of $A$, and $F(q(t))$ is a nonlinear function with a Lipschitz constant bounded independent of $\eps$.
	This oscillatory  second-order system \eqref{oso}
	{recurrently appears in} applied mathematics, molecular biology, chemistry, electronics,
	astronomy, classical mechanics and quantum physics.
	Essentially, all mechanical systems with a partitioned Hamiltonian
	function lend themselves to this form. {One of our important examples
		is discretized partial differential equations (PDEs) such as wave equation, Klein-Gordon equation and sine-Gordon equation for which the method proposed in this paper is
		successfully applied.}
	If we let $z=\begin{pmatrix}
		q\\p
	\end{pmatrix}$ (omit the time variable $t$), then $z$ satisfies the equation \eqref{IVP} with $R(z)\equiv R=\begin{pmatrix}
		O_d&I_d\\
		-\frac{1}{\eps^2}A&O_d
	\end{pmatrix}$ and $f(z)=\begin{pmatrix}
		\textbf{0}_d\\F(q)
	\end{pmatrix}$, where $O_d$ and $I_d$ respectively denote the $d\times d$  zero matrix  and identity matrix, $\textbf{0}_d$ is referred to  a $d$-dimensional zero vector and the initial value is $z_0=\begin{pmatrix}
		q_0\\ p_0
	\end{pmatrix}$.  Furthermore, if the positive semi-definite matrix $A$ is symmetric and the nonlinear term {is defined by} $F(q)=-\nabla V(q)$ with a smooth scalar potential $V(q)$, then
	the model \eqref{oso} becomes a Hamiltonian system  $\dot{z}=J\nabla H_1(z)$ with $J=\begin{pmatrix}
		O_d&I_d\\
		-I_d&O_d
	\end{pmatrix}$  and the Hamiltonian function:
	$$H_1(q,p)=\frac{1}{2}p^\intercal p+\frac{1}{2\eps^2}q^\intercal Aq+V(q).$$

	\textbf{Charged-particle dynamics (CPD).} We now pay attention to the movement of a charged particle in an electromagnetic field, and this is described by the charged-particle dynamics (CPD):\begin{equation}\label{CPD}
		\ddot{x}(t)=\dot{x}(t)\times B(x(t))+E(x(t)),\quad x(0)=x_0,\quad \dot{x}(0)=\dot{x}_0,\quad t\in[0,T],
	\end{equation}
	where {$x(t):[0,T]\to \mathbb{R}^3$ is the unknown position,}
	$B(x(t))=(b_1(x(t)),b_2(x(t)),b_3(x(t)))^\intercal$ is the magnetic field and $E(x(t))$ is the negative gradient of a scalar potential $U(x(t))$. {This kind of system is  interesting for important applications and for example, it appears in Vlasov equation which {is} of paramount importance in tokamak plasmas.} Clearly, letting $v:=\dot{x}$ (omit the time variable $t$) and $y:=\begin{pmatrix}
		x\\v
	\end{pmatrix}$, {the CPD \eqref{CPD}} satisfies the equation \eqref{IVP} with $R(y)=\begin{pmatrix}
		O_3&I_3\\
		O_3&\widehat{B}(x)
	\end{pmatrix}$ and $f(y)=\begin{pmatrix}
		\textbf{0}_3\\E(x)
	\end{pmatrix},$ where $\widehat{B}(x)=\begin{pmatrix}
		0&b_3(x)&-b_2(x)\\
		-b_3(x)&0&b_1(x)\\
		b_2(x)&-b_1(x)&0
	\end{pmatrix}$. {Additionally, as a conservative system, its} {solution exactly {conserves} the  energy}
	\begin{equation}\label{H(x,v)}
		H_2(x,v)=\frac{1}{2}\norm{v}^2+U(x)
	\end{equation}
	with the Euclidean norm  $\norm{\cdot}$.
	
	As stated above, the conservative system \eqref{IVP} admits  energy
	conservation law. Naturally it is
	desirable to propose a numerical method preserving the energy in the discrete sense \cite{06Hairer,21Exponential} and thus energy-preserving (EP) methods which can inherit energy
	conservation law  have gained remarkable success. In recent years, a large amount of effective energy-preserving methods have been developed and analysed for the conservative system \eqref{IVP}  (see, e.g. \cite{12Brugnano,22Long,22Energy,20An,21Exponential,19vp,21Error,13Wu}). However, these energy-preserving methods are all implicit and they require a nonlinear iterative process consuming a lot of time. In order to improve the efficiency of EP methods,
	{some novel technologies have been developed recently.}
	The SAV (scalar auxiliary variable) approach \cite{2018Shen,2019Shen} is developed based on the principle of the energy quadratization as well as the IEQ (invariant energy quadratization) approach \cite{17Yang,17Zhao}, and this yields linearly implicit energy stable schemes. By introducing an auxiliary scalar instead of an auxiliary function in the IEQ  approach, the SAV  strategy is proposed for  constructing energy stable schemes for a broad class of gradient flows \cite{79Allen,58Cahn,04Elder}. This approach offers numerous advantages over conventional energy-preserving techniques, including significant computing efficiency. Therefore,
	it has been applied successfully to many classical models such as
	nonlinear Schrödinger equation \cite{21Akrivis,21Li,22Poulain}, nonlinear Klein-Gordon equation \cite{20Jiang}, Navier–Stokes equation \cite{21Huang,22Zhang}, and so on. However,
	it seems that {the SAV  strategy has not been taken into account}  for  the conservative system \eqref{IVP},
	which motivates this paper.
	
	In this paper, a novel class of linearly implicit energy-preserving schemes is derived and analysed  for the conservative system \eqref{IVP}. The main contributions are {given} as follows. The new methods are formulated based on the splitting of the system and the SAV approach which produce exponential linearly implicit energy-preserving methods. This kind of methods shares the advantages of exponential integrators \cite{16Li, 19Shen} and SAV schemes \cite{2018Shen,2019Shen}. To the best of our knowledge, it seems that this is the first work in the direction
	of designing such SAV-type methods for  the conservative system \eqref{IVP}.
	
	The rest of the paper is organised as follows. In Section \ref{sec2}, we present the formulation of the scheme and its energy-preserving {property} for the OSDE \eqref{oso}.  {The extension of this strategy to  the CPD \eqref{CPD}} is  investigated in  Section \ref{sec3}.  The convergence of the obtained methods  is discussed in Section \ref{sec4}. Section \ref{sec5} carries out five numerical experiments to show the performance and efficiency of the proposed methods.  Finally, some conclusions are drawn in the last section.
	
	\section{Energy-preserving {method} for {OSDE}}\label{sec2}
	In this section, {we formulate the new {method for the OSDE} \eqref{oso}.}
	We start by introducing the notations shown below:
	\begin{equation}
		\delta q^n:=\frac{q^{n+1}-q^n}{h},\quad q^{n+\frac{1}{2}}:=\frac{q^{n+1}+q^n}{2},
	\end{equation}
	{where} $h$ is the step size, $q^n$ refers to the numerical approximation of the solution $q(t_n)$ at $t_n=nh$.
	\subsection{Formulation of the {method}}
	Supposing that the scalar potential $V(q)$ is bounded from below, i.e., there exists $c_0>0$ such that $V(q)\ge -c_0$,  we construct a new function $s(t)=\sqrt{V(q)+C_0}$ with the initial value $s(0)=\sqrt{V(q_0)+C_0}$. Here it is required that $C_0>c_0$ such that $V(q)+C_0\ge C_0-c_0>0$. Then the system \eqref{oso} can be rewritten in the form:
	\begin{equation}\label{new-oso}
		\frac{d}{dt}\begin{pmatrix}
			q\\p\\s	
		\end{pmatrix}=
		\begin{pmatrix}
			p\\
			-\frac{1}{\eps^2}Aq+\frac{F(q)}{\sqrt{V(q)+C_0}}s\\
			-\frac{\dot{q}^\intercal F(q)}{2\sqrt{V(q)+C_0}}
		\end{pmatrix},\quad
		\begin{pmatrix}
			q(0)\\p(0)\\s(0)
		\end{pmatrix}=
		\begin{pmatrix}
			q_0\\p_0\\\sqrt{V(q_0)+C_0}
		\end{pmatrix}.
	\end{equation}
	Let $g(q,s)=\begin{pmatrix}
		-\frac{F(q)}{\sqrt{V(q)+C_0}}s\\\textbf{0}_d
	\end{pmatrix}$, $R=JM$ with $J=\begin{pmatrix}
		O_d&I_d\\-I_d&O_d
	\end{pmatrix}$ and $M=\begin{pmatrix}
		\frac{1}{\eps^2}A&O_d\\O_d&I_d
	\end{pmatrix}$.
	Integrating the equations  \eqref{new-oso} from $t_n$ to $t_{n+1}$ by the variation-of-constants formula, we get
	\begin{equation*}
		\begin{aligned}
			&z(t_{n+1})=\exp(hR)z(t_n)+h\int_{0}^{1}\exp((1-\sigma)hR)Jg(q(t_n+\sigma h),s(t_n+\sigma h))d\sigma,\\
			&s(t_{n+1})=s(t_n)-h\int_{0}^{1}\frac{(\dot{q}(t_n+h\sigma))^\intercal F(q(t_n+h\sigma))}{2\sqrt{V(q(t_n+h\sigma))+C_0}}d\sigma.
		\end{aligned}
	\end{equation*}
	
	Using $g(\tilde{q}^{n+\frac{1}{2}},s^{n+\frac{1}{2}})$ and $\delta q^n$ to replace $g(q(t_n+\sigma h),s(t_n+\sigma h))$ and $\dot{q}(t_n+h\sigma)$ {respectively}, we then arrive at the subsequent linearly implicit scheme:
	\begin{equation}\label{ESAV}
		\begin{aligned}
			&z^{n+1}=\exp(hR)z^n+h\varphi(hR) Jg(\tilde{q}^{n+\frac{1}{2}},s^{n+\frac{1}{2}}),\ \ s^{n+1}=s^n-\frac{(q^{n+1}-q^n)^\intercal F(\tilde{q}^{n+\frac{1}{2}})}{2\sqrt{V(\tilde{q}^{n+\frac{1}{2}})+C_0}},
		\end{aligned}
	\end{equation}
	{where $\tilde{q}^{n+\frac{1}{2}}:=\begin{pmatrix}
			I_d&O_d
		\end{pmatrix} \frac{1}{2}\left( I_{2d}+\exp(hR)\right)z^n $ is the approximation of $q^{n+\frac{1}{2}}$ and the scalar function is given by $\varphi(z)=\int_{0}^{1}\exp((1-\sigma)z)d\sigma=(\exp(z)-1)/z$.}
	
	Next, we describe how the aforementioned scheme can be implemented efficiently.
	{According to} the scheme of $R$,  $\exp(hR)$ and $\varphi(hR)$ can be partitioned into
	$$\exp(hR)=\begin{pmatrix}
		\cos(h\sqrt{A}/\eps)&h\sinc(h\sqrt{A}/\eps)\\
		-\sqrt{A}/\eps\sin(h\sqrt{A}/\eps)&\cos(h\sqrt{A}/\eps)
	\end{pmatrix},\quad \varphi(hR)=\begin{pmatrix}
		\sinc(h\sqrt{A}/\eps)&hg_1(h\sqrt{A}/\eps)\\
		h^{-1}g_2(h\sqrt{A}/\eps)&\sinc(h\sqrt{A}/\eps)
	\end{pmatrix},$$
	with functions $\sinc(z)=\sin(z)/z$, $g_1(z)=(1-\cos(z))/z^2$ and $g_2(z)=\cos(z)-1.$
	{Then we rewrite} \eqref{ESAV} as
	\begin{align}
		q^{n+1}=&\cos(h\sqrt{A}/\eps)q^n+h\sinc(h\sqrt{A}/\eps)p^n+h^2g_1(h\sqrt{A}/\eps)\frac{ F(\tilde{q}^{n+\frac{1}{2}})}{\sqrt{V(\tilde{q}^{n+\frac{1}{2}})+C_0}}s^{n+\frac{1}{2}}\label{q^n},\\
		p^{n+1}=&-\sqrt{A}/\eps\sin(h\sqrt{A}/\eps)q^n+\cos(h\sqrt{A}/\eps)p^n+h\sinc(h\sqrt{A}/\eps)\frac{ F(\tilde{q}^{n+\frac{1}{2}})}{\sqrt{V(\tilde{q}^{n+\frac{1}{2}})+C_0}}s^{n+\frac{1}{2}}\label{p^n},\\
		s^{n+1}=&s^n-\frac{(q^{n+1}-q^n)^\intercal F(\tilde{q}^{n+\frac{1}{2}})}{2\sqrt{V(\tilde{q}^{n+\frac{1}{2}})+C_0}}\label{s^n}.
	\end{align}
	By substituting \eqref{s^n} into \eqref{q^n}, we obtain the following linear equation
	\begin{equation}\label{l_n}
		q^{n+1}+\gamma_n\left(F(\tilde{q}^{n+\frac{1}{2}})\right) ^\intercal q^{n+1}=l_n,
	\end{equation}
	where $\gamma_n=\frac{h^2g_1(h\sqrt{A}/\eps)F(\tilde{q}^{n+\frac{1}{2}})}{4V(\tilde{q}^{n+\frac{1}{2}})+4C_0}$ and { $$l_n=\cos(h\sqrt{A}/\eps)q^n+h\sinc(h\sqrt{A}/\eps)p^n+h^2g_1(h\sqrt{A}/\eps)\frac{ F(\tilde{q}^{n+\frac{1}{2}})}{\sqrt{V(\tilde{q}^{n+\frac{1}{2}})+C_0}}s^n+\gamma_n \left( F(\tilde{q}^{n+\frac{1}{2}})\right) ^\intercal q^n.$$}
	
	As we can see, the crucial step is to determine $\left(  F(\tilde{q}^{n+\frac{1}{2}})\right) ^\intercal q^{n+1}$ from the above equation. To this end, taking the inner product of \eqref{l_n} with $F(\tilde{q}^{n+\frac{1}{2}})$, we obtain
	$$\left[ 1+\left( F(\tilde{q}^{n+\frac{1}{2}})\right) ^\intercal \gamma_n\right] \left( F(\tilde{q}^{n+\frac{1}{2}})\right) ^\intercal q^{n+1}=\left( F(\tilde{q}^{n+\frac{1}{2}})\right) ^\intercal l_n,$$
	where  $1+\left( F(\tilde{q}^{n+\frac{1}{2}})\right) ^\intercal \gamma_n> 0$, {since $h^2g_1(h\sqrt{A}/\eps)$ is a symmetrical positive semi-definite matrix.} Consequently,
	\begin{equation}\label{Fq}
		\left( F(\tilde{q}^{n+\frac{1}{2}})\right) ^\intercal q^{n+1}=\frac{\left( F(\tilde{q}^{n+\frac{1}{2}})\right) ^\intercal l_n}{1+\left( F(\tilde{q}^{n+\frac{1}{2}})\right) ^\intercal \gamma_n},
	\end{equation}
	and then we can obtain $q^{n+1}$ from the above equation and the linear equation \eqref{l_n}. Subsequently, $s^{n+1}$ is obtained from \eqref{s^n}. Finally, we get $p^{n+1}$ by substituting $s^{n+1}$ into \eqref{p^n}.  As a result, the total cost at each time step comes essentially from solving two linear equations $\left( F(\tilde{q}^{n+\frac{1}{2}})\right) ^\intercal l_n$ and $\left( F(\tilde{q}^{n+\frac{1}{2}})\right) ^\intercal q^{n+1}$ using \eqref{Fq} with constant coefficients. This demonstrates that the scheme is absolutely efficient and straightforward to implement. This new scheme \eqref{q^n}-\eqref{s^n} will be referred as E2-SAV.
	
	\subsection{Energy-preserving property}
	{We first present an energy-preserving property of the new system \eqref{new-oso}.}
	\begin{lem}
		The system \eqref{new-oso} conserves the following modified energy:
		$$\widehat{H}(q,p,s)=\frac{1}{2}p^\intercal p+\frac{1}{2\eps^2}q^\intercal Aq+s^2-C_0=\widehat{H}(q_0,p_0,s_0).$$
	\end{lem}
	\begin{proof}{Multiplying $p^\intercal$ with the second equation in system \eqref{new-oso} gives}
		$$\frac{d}{dt}\left( \frac{1}{2}p^\intercal p\right)=p^\intercal \dot{p}=p^\intercal\left( -\frac{1}{\eps^2}Aq+\frac{F(q)}{\sqrt{V(q)+C_0}}s\right)=-\frac{d}{dt}(\frac{1}{2\eps^2}q^\intercal Aq+s^2-C_0),$$
		which completes the proof.
	\end{proof}
	
	{Before deriving the energy-preserving property of E2-SAV, the following proposition is needed.}
	\begin{lem}\label{null}
		For any symmetric matrix $M$ and scalar $h\ge0$, {it is true that
			$$\exp(hJM)^\intercal M\exp(hJM)=M,$$  where} $J$ is skew symmetric.
	\end{lem}
	
	The proof of this lemma is straightforward and we skip it for brevity. With these preparations,   our new method has the following energy conservation property.
	\begin{mytheo}\label{ep}
		The linearly implicit scheme E2-SAV exactly preserves the modified energy $\widehat{H}(q,p,s)$ at discrete level, i.e., {for $n=0,1,\ldots,T/h-1,$}
		$$\widehat{H}_h^{n+1}=\widehat{H}_h^n \quad {\textmd{with}} \ \ \widehat{H}_h^n=\widehat{H}(q^n,p^n,s^n)=\frac{1}{2}(p^n)^\intercal p^n+\frac{1}{2\eps^2}(q^n)^\intercal Aq^n+(s^n)^2-C_0. $$
	\end{mytheo}
	\begin{proof}
		{Although  the relationship $s=\sqrt{V(q)+C_0}$ is true,  for numerical results $s^n$ and $q^n$, the corresponding result $s^n=\sqrt{V(q^n)+C_0}$ usually does not hold anymore.} This demonstrates that the modified energy $\widehat{H}(q^n,p^n,s^n)$ is  distinct  from the original Hamiltonian energy $H_1(q^n,p^n)$. When $A$ is positive definite, $M=\begin{pmatrix}
			\frac{1}{\eps^2}A&O_d\\O_d&I_d
		\end{pmatrix}$ is nonsingular. Denoting $\tilde{g}=M^{-1}g(\tilde{q}^{n+\frac{1}{2}},s^{n+\frac{1}{2}})$,   we have $h\varphi(hR) Jg(\tilde{q}^{n+\frac{1}{2}},s^{n+\frac{1}{2}})=\varphi(hR)hR\tilde{g}=\left( \exp(hR)-I\right) \tilde{g}$ and
		\begin{align*}
			\frac{1}{2}\left( z^{n+1}\right) ^\intercal Mz^{n+1}=&\frac{1}{2}\left[(z^n)^\intercal \exp(hR)^\intercal+\tilde{g}^\intercal (\exp(hR)-I)^\intercal \right] M\left[ \exp(hR)z^n+\left( \exp(hR)-I\right) \tilde{g}\right] \\
			=&\frac{1}{2}(z^n)^\intercal \exp(hR)^\intercal M \exp(hR)z^n+(z^n)^\intercal \exp(hR)^\intercal M \left( \exp(hR)-I\right) \tilde{g}\\
			&+\frac{1}{2}\tilde{g}^\intercal (\exp(hR)-I)^\intercal M \left( \exp(hR)-I\right) \tilde{g}.
		\end{align*}
		{On the other hand}, it follows from $s^{n+1}$ in \eqref{ESAV} that
		\begin{align*}
			&(s^{n+1})^2-(s^n)^2=-\frac{(q^{n+1}-q^n)^\intercal F(\tilde{q}^{n+\frac{1}{2}})}{\sqrt{V(\tilde{q}^{n+\frac{1}{2}})+C_0}}s^{n+\frac{1}{2}}=\left( (z^{n+1})^\intercal -(z^n)^\intercal\right)M\tilde{g}\\
			&=\left[ (z^n)^\intercal(\exp(hR)-I)^\intercal+\tilde{g}^\intercal(\exp(hR)-I)^\intercal\right] M\tilde{g}
			=((z^n)^\intercal+\tilde{g}^\intercal)(\exp(hR)-I)^\intercal M\tilde{g}.
		\end{align*}
		Accordingly, we can deduce from the above results that
		\begin{align*}
			&\widehat{H}_h^{n+1}-\widehat{H}_h^n=\frac{1}{2}\left( z^{n+1}\right) ^\intercal Mz^{n+1}-\frac{1}{2}\left( z^n\right) ^\intercal Mz^n+(s^{n+1})^2-(s^n)^2\\
			&=\frac{1}{2}(z^n+\tilde{g})^\intercal [\exp(hR)^\intercal M\exp(hR)-M](z^n+\tilde{g})+\frac{1}{2}\tilde{g}^\intercal[\exp(hR)^\intercal M-M\exp(hR)]\tilde{g}
			=0,
		\end{align*}
		with {the result proposed} in Lemma \ref{null} and the skew-symmetry of $\exp(hR)^\intercal M-M\exp(hR)$.
		
		In another case where $A$ is positive semi-definite, then $M$ {might be} singular. There is no doubt that we can find a series of symmetric and nonsingular matrices
		${M_\alpha}$ which converge to $M$ when $\alpha\to 0$. Let $z_\alpha^n$ and $q_\alpha^n$ satisfy
		$$z_\alpha^{n+1}=\exp(hR_\alpha)z_\alpha^n+h\varphi(hR_\alpha) Jg(\tilde{q}_\alpha^{n+\frac{1}{2}},s_\alpha^{n+\frac{1}{2}}),\ \ s_\alpha^{n+1}=s_\alpha^n-\frac{(q_\alpha^{n+1}-q_\alpha^n)^\intercal F(\tilde{q}_\alpha^{n+\frac{1}{2}})}{2\sqrt{V(\tilde{q}_\alpha^{n+\frac{1}{2}})+C_0}}.$$
		Therefore, it still holds that
		$$\widehat{H}_\alpha(z_\alpha^{n+1},s_\alpha^{n+1})-\widehat{H}_\alpha(z_\alpha^n,s_\alpha^n)=\frac{1}{2}\left( z_\alpha^{n+1}\right) ^\intercal M_\alpha z_\alpha^{n+1}-\frac{1}{2}\left( z_\alpha^n\right) ^\intercal M_\alpha z_\alpha^n+(s_\alpha^{n+1})^2-(s_\alpha^n)^2=0.$$ In the end, when
		$\alpha\to 0$,  {we have $z_\alpha^n\to z^n$, $s_\alpha^n\to s^n$, and further} $\widehat{H}_h^{n+1}=\widehat{H}_h^n$.
	\end{proof}
	\begin{rem}
		It should be pointed out that the discrete modified energy-preserving property in Theorem \ref{ep} does not depend on the approximate term $\tilde{q}^{n+\frac{1}{2}}$ for $q^{n+\frac{1}{2}}$. This indicates that {various} numerical approximations for $q^{n+\frac{1}{2}}$ can be chosen without affecting its energy-preserving property. 
	\end{rem}
	
	\section{Extension to charged-particle dynamics}\label{sec3}
	In this part, we concentrate on constructing the numerical solutions of the CPD by combining the ideas of splitting and {the scalar auxiliary variable {in Section \ref{sec2}}.}
	\subsection{Numerical methods}Assuming that the scalar potential $U(x)$ is bounded from below, i.e., there exists $c_0>0$ such that $U(x)\ge -c_0$, we introduce a scalar auxiliary variable $r(t)=\sqrt{U(x)+C_0}$ with $r(0)=\sqrt{U(x_0)+C_0}:=r_0$ and $C_0>c_0$. Then the CPD \eqref{CPD} is equivalent to
	\begin{equation}\label{new CPD}
		\frac{d}{dt}\begin{pmatrix}
			x\\v\\r
		\end{pmatrix}=
		\begin{pmatrix}
			v\\
			\widehat{B}(x)v+\frac{E(x)}{\sqrt{U(x)+C_0}}r\\
			-\frac{\dot{x}^\intercal E(x)}{2\sqrt{U(x)+C_0}}
		\end{pmatrix},\quad
		\begin{pmatrix}
			x(0)\\v(0)\\r(0)
		\end{pmatrix}=
		\begin{pmatrix}
			x_0\\v_0\\r_0
		\end{pmatrix},
	\end{equation}
	where $v_0:=\dot{x}_0$.
	In order to get the numerical solution of the system \eqref{new CPD}, we split it into two subflows:
	\begin{equation}\label{splitting}
		\frac{d}{dt}\begin{pmatrix}
			x\\v\\r
		\end{pmatrix}=\begin{pmatrix}
			0\\\widehat{B}(x)v\\0
		\end{pmatrix},\quad\frac{d}{dt}\begin{pmatrix}
			x\\v\\r
		\end{pmatrix}=\begin{pmatrix}
			v\\
			\frac{E(x)}{\sqrt{U(x)+C_0}}r\\
			-\frac{\dot{x}^\intercal E(x)}{2\sqrt{U(x)+C_0}}
		\end{pmatrix}.
	\end{equation}
	For the first subflow, it is easy to derive its exact solution
	$\Phi_t^L:\begin{pmatrix}
		x(t)\\v(t)\\r(t)
	\end{pmatrix}=\begin{pmatrix}
		x(0)\\e^{t\widehat{B}(x(0))}v(0)\\r(0)
	\end{pmatrix}.$
	Since the second subflow without the scalar $r(t)$ is a canonical Hamiltonian system, we can apply the E2-SAV \eqref{q^n}-\eqref{s^n} to this subflow to get a linearly implicit numerical propagator $\Phi^{NL}_h$:
	\begin{align}
		&x^{n+1}=x^n+hv^n+\frac{h^2}{2}\frac{E(\widehat{x}^{n+\frac{1}{2}})}{\sqrt{U(\widehat{x}^{n+\frac{1}{2}})+C_0}}r^{n+\frac{1}{2}}\label{x^n},\\
		&v^{n+1}=v^n+h\frac{E(\widehat{x}^{n+\frac{1}{2}})}{\sqrt{U(\widehat{x}^{n+\frac{1}{2}})+C_0}}r^{n+\frac{1}{2}}\label{v^n},\\
		&r^{n+1}
		=r^n-\frac{(x^{n+1}-x^n)^\intercal E(\widehat{x}^{n+\frac{1}{2}})}{2\sqrt{U(\widehat{x}^{n+\frac{1}{2}})+C_0}},\label{r^n}
	\end{align}
	with approximate term $\widehat{x}^{n+\frac{1}{2}}=x^n+\frac{h}{2}v^n$.
	
	In this manner, the {propagator $\Phi^{NL}_h$} naturally enjoys the high efficiency as follows. First, we can obtain $x^{n+1}$ by using the notation $r^{n+\frac{1}{2}}=\frac{r^{n+1}+r^n}{2}$ and \eqref{r^n} to eliminate the implicit term  $r^{n+1}$ in \eqref{x^n}. Secondly, $r^{n+1}$ is {formulated} from \eqref{r^n}. Finally, we get $v^{n+1}$ by substituting $r^{n+1}$ into \eqref{v^n}. In conclusion, the total cost to get the numerical solution is only to solve a linear equation \eqref{x^n} with constant coefficients. Indeed, we obtain the following explicit expression of $\Phi^{NL}_h$:
	{\begin{equation*}
			\begin{aligned}
				&x^{n+1}=x^n+A_n\left(hv^n+\frac{h^2}{2}\frac{E(\widehat{x}^{n+\frac{1}{2}})}{\sqrt{U(\widehat{x}^{n+\frac{1}{2}})+C_0}}r^n\right),\\
				&v^{n+1}=B_nv^n+c_nh\frac{E(\widehat{x}^{n+\frac{1}{2}})}{\sqrt{U(\widehat{x}^{n+\frac{1}{2}})+C_0}}r^n,\;\ \ \ r^{n+1}=b_nr^n-\frac{hE(\widehat{x}^{n+\frac{1}{2}})^\intercal A_n}{2\sqrt{U(\widehat{x}^{n+\frac{1}{2}})+C_0}}v^n,
			\end{aligned}
		\end{equation*}
		where the matrices are  $A_n=I_3-\frac{h^2}{8a_n}\frac{E(\widehat{x}^{n+\frac{1}{2}})E(\widehat{x}^{n+\frac{1}{2}})^\intercal}{U(\widehat{x}^{n+\frac{1}{2}})+C_0}$, $B_n=I_3-\frac{h^2}{4}\frac{E(\widehat{x}^{n+\frac{1}{2}})E(\widehat{x}^{n+\frac{1}{2}})^\intercal A_n}{U(\widehat{x}^{n+\frac{1}{2}})+C_0}$, and the constants are $a_n=1+\frac{h^2}{8}\frac{\abs{E(\widehat{x}^{n+\frac{1}{2}})}^2}{U(\widehat{x}^{n+\frac{1}{2}})+C_0}$, $b_n=1-\frac{h^2}{4}\frac{E(\widehat{x}^{n+\frac{1}{2}})^\intercal A_nE(\widehat{x}^{n+\frac{1}{2}})}{U(\widehat{x}^{n+\frac{1}{2}})+C_0}$, $c_n=\frac{b_n+1}{2}$.}
	
	{In the light of  the above  preparations, we are now} in the position to present our methods.
	\begin{algo}\label{SAVs}
		Supposing that the numerical solutions are $x^n\approx x(t_n)$, $v^n \approx v(t_n)$ and choosing {the initial values as}
		$x^0 = x_0$, $v^0 = v_0$, then through composition of $\Phi^{NL}_h$ and $\Phi^L_h$, we obtain the full scheme such as $\Phi^{S1}_h=\Phi^{NL}_h\circ \Phi^L_h$
		, which is known as the first order splitting scheme. {More specifically, this iterative} scheme for solving \eqref{new CPD} reads as: for $n \ge 0$,
		\begin{equation}\label{im-S1-SAV}
			\begin{aligned}
				&x^{n+1}=x^n+he^{h\widehat{B}(x^n)}v^n+\frac{h^2}{2}\frac{E(\tilde{x}^{n+\frac{1}{2}})}{\sqrt{U(\tilde{x}^{n+\frac{1}{2}})+C_0}}r^{n+\frac{1}{2}},\\
				&v^{n+1}=e^{h\widehat{B}(x^n)}v^n+h\frac{E(\tilde{x}^{n+\frac{1}{2}})}{\sqrt{U(\tilde{x}^{n+\frac{1}{2}})+C_0}}r^{n+\frac{1}{2}},
				\;\ \ r^{n+1}
				=r^n-\frac{(x^{n+1}-x^n)^\intercal E(\tilde{x}^{n+\frac{1}{2}})}{2\sqrt{U(\tilde{x}^{n+\frac{1}{2}})+C_0}},
			\end{aligned}
		\end{equation}
		where $\tilde{x}^{n+\frac{1}{2}}=x^n+\frac{h}{2}e^{h\widehat{B}(x^n)}v^n.$	
		We shall refer to this scheme as S1-SAV.
		
		Moreover, we can take symmetric Strang splitting: $$\Phi^{S2}_h=\Phi^L_\frac{h}{2}\circ\Phi^{NL}_h\circ \Phi^L_\frac{h}{2},$$ called S2-SAV later. {Now, we will describe how the symmetric Strang splitting scheme can be used to construct linearly implicit methods of higher order by Triple Jump splitting. To begin with, we get: $$\Phi^{S4}_h=\Phi^{S2}_{\tau_1h}\circ\Phi^{S2}_{\tau_2h}\circ \Phi^{S2}_{\tau_3h},$$ where $\tau_1=\tau_3=\frac{1}{2-2^{1/3}}$ and  $\tau_2=-\frac{2^{1/3}}{2-2^{1/3}}$ satisfy $\sum_{i=1}^{3}\tau_i=1$ and $\sum_{i=1}^{3}(\tau_i)^3=0$, and we will refer to this scheme as S4-SAV. And then} a new scheme S6-SAV of order-6 can be constructed as: $$\Phi^{S6}_h=\Phi^{S4}_{\theta_1h}\circ\Phi^{S4}_{\theta_2h}\circ \Phi^{S4}_{\theta_3h},$$ where $\theta_1=\theta_3=\frac{1}{2-2^{1/5}}$ and  $\theta_2=-\frac{2^{1/5}}{2-2^{1/5}}$ satisfy $\sum_{i=1}^{3}\theta_i=1$ and $\sum_{i=1}^{3}(\theta_i)^5=0$. Without loss of generality, by repeating this process, we can get  linearly implicit methods with arbitrary even high order.
	\end{algo}

	\begin{rem}
		It should be noted that the method S1-SAV can be expressed in an explicit form:
	{\begin{equation}\label{ex-S1-SAV}
			\begin{aligned} &x^{n+1}=x^n+\tilde{A}_n\left(he^{h\widehat{B}(x^n)}v^n+\frac{h^2}{2}\frac{E(\tilde{x}^{n+\frac{1}{2}})}{\sqrt{U(\tilde{x}^{n+\frac{1}{2}})+C_0}}r^n\right),\\	&v^{n+1}=\tilde{B}_ne^{h\widehat{B}(x^n)}v^n+\tilde{c}_nh\frac{E(\tilde{x}^{n+\frac{1}{2}})}
				{\sqrt{U(\tilde{x}^{n+\frac{1}{2}})+C_0}}r^n,\;\ \ \ r^{n+1}=\tilde{b}_nr^n-\frac{hE(\tilde{x}^{n+\frac{1}{2}})
					^\intercal \tilde{A}_n}{2\sqrt{U(\tilde{x}^{n+\frac{1}{2}})+C_0}}e^{h\widehat{B}(x^n)}v^n,
			\end{aligned}
		\end{equation}
		where the matrices are  $\tilde{A}_n=I_3-\frac{h^2}{8\tilde{a}_n}\frac{E(\tilde{x}^{n+\frac{1}{2}})E(\tilde{x}^{n+\frac{1}{2}})^\intercal}{U(\tilde{x}^{n+\frac{1}{2}})+C_0}$, $\tilde{B}_n=I_3-\frac{h^2}{4}\frac{E(\tilde{x}^{n+\frac{1}{2}})E(\tilde{x}^{n+\frac{1}{2}})^\intercal A_n}{U(\tilde{x}^{n+\frac{1}{2}})+C_0}$, and the constants are  $\tilde{a}_n=1+\frac{h^2}{8}\frac{\abs{E(\tilde{x}^{n+\frac{1}{2}})}^2}{U(\tilde{x}^{n+\frac{1}{2}})+C_0}$,  $\tilde{b}_n=1-\frac{h^2}{4}\frac{E(\tilde{x}^{n+\frac{1}{2}})^\intercal \tilde{A}_nE(\tilde{x}^{n+\frac{1}{2}})}{U(\tilde{x}^{n+\frac{1}{2}})+C_0}$, $\tilde{c}_n=\frac{\tilde{b}_n+1}{2}$.}
		The expressions of other schemes are similar to S1-SAV but of complicated form,  and thus we omit them for brevity. In addition, we shall abbreviate the above class of linearly implicit splitting schemes as SSAVs.
	\end{rem}

	\subsection{Energy-preserving properties}
	In this subsection, we focus on the energy-preserving properties of  SSAVs for the CPD in Algorithm  \ref{SAVs}.
	\begin{lem}\label{me}
		The second subflow in \eqref{splitting} conserves the following modified energy:
		$$\tilde{H}(v,r)=\frac{1}{2}\norm{v}^2+r^2-C_0=\tilde{H}(v_0,r_0).$$
	\end{lem}
	\begin{proof}
		Taking the inner {product} with $v$ of the second  equation in  the second subflow, we have $$\frac{d}{dt}\left(\frac{1}{2}\norm{v}^2 \right) =v^\intercal\dot{v}=-\frac{v^\intercal\nabla U(x)}{\sqrt{U(x)+C_0}}r=-2\dot{r}r=-\frac{d}{dt}r^2=-\frac{d}{dt}\left( r^2-C_0\right).$$
		{Thence, it is obtained that} $\frac{d}{dt}\tilde{H}(v,r)=0.$ Actually, $\tilde{H}(v,r)=H_2(x,v)$, {and thus} the second subflow exactly preserves the energy \eqref{H(x,v)}.
	\end{proof}
	
	\begin{mytheo}\label{energy-preserving}
		Algorithm \ref{SAVs} {exactly} preserves the following discrete modified energy:
		$$\tilde{H}^{n+1}_h=\tilde{H}^n_h\quad {\textmd{with}} \ \ \tilde{H}^n_h=\tilde{H}(v^n,r^n)=\frac{1}{2}\norm{v^n}^2+(r^n)^2-C_0,$$
		{where $n=0,1,\ldots,T/h-1$.}
	\end{mytheo}
	\begin{proof}
		{On the basis of the fact that the propagator $\Phi^L_t$ is the exact solution operator of the first subflow, it follows that} it preserves the energy \eqref{H(x,v)}. {By Lemma \ref{me} and noticing} that $\tilde{H}(v,r)=H_2(x,v)$, {one gets that the propagator $\Phi^L_t$ exactly preserves the energy $\tilde{H}(v,r)$.} On the other hand, for the propagator $\Phi_h^{NL}$, using Theorem \ref{ep}, one has $\tilde{H}^{n+1}_h=\tilde{H}^n_h$.
		
		According to Algorithm \ref{SAVs}, we  are aware  that  all the proposed  methods are constructed by composing $\Phi^{NL}_h$ and $\Phi^L_h$. Therefore, the energy conservation  $\tilde{H}^{n+1}_h=\tilde{H}^n_h$ holds for S1-SAV, S2-SAV, S4-SAV and S6-SAV.
	\end{proof}
	\section{Convergence}\label{sec4}
	In this section, we provide rigorous error estimates of our methods for CPD. With the same arguments,  the proof is easily presented for the method E2-SAV applied to OSDE and we omit it  for  brevity.
	For the sake of simplicity, we only consider the CPD in a constant magnetic, i.e. $B(x)\equiv B$ and the nonlinear term $E(x)$ satisfies $E(\textbf{0}_3)=\textbf{0}_3$. For non-homogeneous magnetic fields,  a  linearized system can be considered and then by deriving the errors between the original and linearized systems,  the convergence can be transformed to be studied for the system with a constant magnetic.
	Throughout this section, $\norm{\cdot}$ denotes the Euclidean norm in finite dimensional space and $\abs{\cdot}$ refers to the absolute value of a function.
	In what follows, to be precise with the methods, we
	consider here S1-SAV as an example.
	
	\begin{lem}[Local boundedness of numerical solutions]\label{local bounded}
		{For the nonlinear function  $E: \mathbb{R}^3 \rightarrow \mathbb{R}^3$, it is assumed that it is sufficient differentiable and satisfies Lipschitz condition, i.e., there exists $L>0$ such that $\norm{E(x(t))-E(\tilde{x}(t))} \le L\norm{x(t)-\tilde{x}(t)}$ for all $t\in [0,T]$. There exists a sufficient small $0<\beta \le 1$, such that  if $0<h\le \beta$ and the numerical solutions S1-SAV at $t_n$ is bounded, i.e., {$\norm{x^n}\le K, \norm{v^n}\le K, \abs{r^n}\le K$} for $K\ge 0$, then we have
			$$\norm{x^{n+1}}\le C_K,\quad\norm{v^{n+1}}\le C_K,\quad\abs{r^{n+1}}\le C_K, $$
			where $C_K$ is independent of  the stepsize $h$ and $n$.}
	\end{lem}
	\begin{proof}
		Using Minkowski’s inequality to the explicit form \eqref{ex-S1-SAV} of S1-SAV, we can get
		\begin{align*}
			&\norm{x^{n+1}}\le \norm{x^n}+h\norm{e^{h\widehat{B}}}\norm{v^n}+\frac{h^2}{2}\norm{\frac{E(\tilde{x}^{n+\frac{1}{2}})}{\sqrt{U(\tilde{x}^{n+\frac{1}{2}})+C_0}}}\abs{r^n},\\
			&\norm{v^{n+1}}\le \norm{e^{h\widehat{B}}}\norm{v^n}+h\norm{\frac{E(\tilde{x}^{n+\frac{1}{2}})}{\sqrt{U(\tilde{x}^{n+\frac{1}{2}})+C_0}}}\abs{r^n},\\
			&\abs{r^{n+1}}\le \abs{r^n}+\frac{h}{2}\norm{\frac{E(\tilde{x}^{n+\frac{1}{2}})}{\sqrt{U(\tilde{x}^{n+\frac{1}{2}})+C_0}}}\norm{e^{h\widehat{B}}}\norm{v^n},
		\end{align*}
		{where the fact that $0<\norm{\tilde{A}_n}, \norm{\tilde{B}_n}, \tilde{b}_n, \tilde{c}_n<1$ is used here.} 
		Summing up these equations and using  the facts $\norm{e^{h\widehat{B}}}=1$, $\sqrt{U(x)+C_0}\ge \sqrt{C_0-c_0}$ and $\norm{E(x)}=\norm{E(x)-E(\textbf{0}_3)}\le L\norm{x}$ {lead} to
		\begin{align*}
			&\norm{x^{n+1}}+\norm{v^{n+1}}+\abs{r^{n+1}}\le \norm{x^n}+h\norm{e^{h\widehat{B}}}\norm{v^n}+\norm{e^{h\widehat{B}}}\norm{v^n}+\abs{r^n}\\
			&+\frac{h^2}{2}\norm{\frac{E(\tilde{x}^{n+\frac{1}{2}})}{\sqrt{U(\tilde{x}^{n+\frac{1}{2}})+C_0}}}\abs{r^n}+h\norm{\frac{E(\tilde{x}^{n+\frac{1}{2}})}{\sqrt{U(\tilde{x}^{n+\frac{1}{2}})+C_0}}}\abs{r^n}+\frac{h}{2}\norm{\frac{E(\tilde{x}^{n+\frac{1}{2}})}{\sqrt{U(\tilde{x}^{n+\frac{1}{2}})+C_0}}}\norm{e^{h\widehat{B}}}\norm{v^n}\\
			&\le (1+h) (\norm{x^n}+\norm{v^n}+\abs{r^n})+{\frac{2h}{\sqrt{C_0-c_0}}}\norm{E(\tilde{x}^{n+\frac{1}{2}})}(\norm{v^n}+\abs{r^n})\\
			&\le (1+h) (\norm{x^n}+\norm{v^n}+\abs{r^n})+{\frac{2h}{\sqrt{C_0-c_0}}}L\left( \norm{x^n}+\frac{h}{2}\norm{e^{h\widehat{B}}}\norm{v^n}\right) (\norm{v^n}+\abs{r^n})\\
			&\le (1+h) (\norm{x^n}+\norm{v^n}+\abs{r^n})+{\frac{2h}{\sqrt{C_0-c_0}}}L( \norm{x^n}+\norm{v^n}) (\norm{v^n}+\abs{r^n})\le C_K.
		\end{align*}
	\end{proof}
	
	\begin{lem}[Global boundedness of numerical solutions]\label{global bounded}
		{Under the conditions of Lemma \ref{local bounded} and  the requirement that {$\norm{x^0}\le \widehat{K}, \norm{v^0}\le \widehat{K}, \abs{r^0}\le \widehat{K}$} for $\widehat{K}\ge 0$, the numerical solution produced by  S1-SAV  is bounded as
			$$\norm{x^{n+1}}\le C_{\widehat{K}},\quad\norm{v^{n+1}}\le C_{\widehat{K}},\quad\abs{r^{n+1}}\le C_{\widehat{K}},$$
			where $C_{\widehat{K}}$ is independent of  the stepsize $h$ and $n$.}
	\end{lem}
	\begin{proof}
		By mathematical induction, the proof is quite similar to Lemma \ref{local bounded} and therefore we leave out the details to keep concise.
	\end{proof}
	\begin{mytheo}
		Under the  assumptions in Lemma \ref{local bounded} {and Lemma \ref{global bounded}}, supposing that the CPD \eqref{CPD} has sufficiently smooth solutions, then the convergence of the scheme \eqref{im-S1-SAV} is given by
		$$\norm{x(t_n)-x^n}\le Ch,\quad \norm{v(t_n)-v^n}\le Ch,$$
		where $C$ is a general constant independent of the stepsize $h$ {and $n$ but depends on the interval length $T$, the bound of numerical {solutions} $C_{\widehat{K}}$ and the Lipschitz constant $L$.}
	\end{mytheo}
	\begin{proof}
		(I) First, we formulate the truncation errors by inserting the exat solution into \eqref{im-S1-SAV} as follows:
		\begin{equation}\label{truncation}
			\begin{aligned}
				&x(t_{n+1})=x(t_n)+he^{h\widehat{B}}v(t_n)+\frac{h^2}{2}\frac{E(x(t_n)+\frac{h}{2}e^{h\widehat{B}}v(t_n))}{\sqrt{U(x(t_n)+\frac{h}{2}e^{h\widehat{B}}v(t_n))+C_0}}\frac{r(t_{n+1})+r(t_n)}{2}+R_x^{n+1},\\
				&v(t_{n+1})=e^{h\widehat{B}}v(t_n)+h\frac{E(x(t_n)+\frac{h}{2}e^{h\widehat{B}}v(t_n))}{\sqrt{U(x(t_n)+\frac{h}{2}e^{h\widehat{B}}v(t_n))+C_0}}\frac{r(t_{n+1})+r(t_n)}{2}+R_v^{n+1},\\
				&r(t_{n+1})=r(t_n)-\left( x(t_{n+1})-x(t_n)\right) ^\intercal\frac{E(x(t_n)+\frac{h}{2}e^{h\widehat{B}}v(t_n))}
				{2\sqrt{U(x(t_n)+\frac{h}{2}e^{h\widehat{B}}v(t_n))+C_0}}+R_r^{n+1},
			\end{aligned}
		\end{equation}
		{where $R_x^{n+1},R_v^{n+1},R_r^{n+1}$ are the deviations.} Then we get
		\begin{equation*}
			R_w^{n+1}=\Phi_h(w(t_n))-\Phi^{NL}_h\circ \Phi^L_h(w(t_n)),\\
		\end{equation*}
		{in which} $\Phi_h$ denotes the exact flow of \eqref{new CPD}, $w(t)=\begin{pmatrix}
			x(t)\\
			v(t)\\
			r(t)
		\end{pmatrix}$ and
		$R_w^{n+1}=\begin{pmatrix}
			R_x^{n+1}\\
			R_v^{n+1}\\
			R_r^{n+1}
		\end{pmatrix}$.
		
		By Taylor expansion, we reach
		$$\Phi_h(w)=w+hf(w)+\mathcal{O}(h^2),\;\Phi_h^L(w)=w+hf_1(w)+\mathcal{O}(h^2),\;\Phi_h^{NL}(w)=w+hf_2(w)+\mathcal{O}(h^2)$$
		with $f(w)=\begin{pmatrix}
			\textbf{0}_3\\\widehat{B}(x)v\\0
		\end{pmatrix}+\begin{pmatrix}
			v\\
			\frac{E(x)}{\sqrt{U(x)+C_0}}r\\
			-\frac{\dot{x}^\intercal E(x)}{2\sqrt{U(x)+C_0}}
		\end{pmatrix}:=f_1(w)+f_2(w)$.
		Hence, one gets
		\begin{align*}
			\widehat{\Phi}^{NL}_h\circ \Phi^L_h(w)&=w+hf_1(w)+\mathcal{O}(h^2)+hf_2(w+hf_1(w)+\mathcal{O}(h^2))+\mathcal{O}(h^2)\\
			&=w+hf_1(w)+hf_2(w)+\mathcal{O}(h^2)=\Phi_h(w)+\mathcal{O}(h^2),
		\end{align*}
		which gives that $R_w^{n+1}=\mathcal{O}(h^2)$.
		
		(II) Secondly, we define the following error functions:
		$$e_x^n=x(t_n)-x^n,\; e_v^n=v(t_n)-v^n,\; e_r^n=r(t_n)-r^n.$$
		For the sake of brevity, {we denote $\widehat{E}(t):=E(x(t)+\frac{h}{2}e^{h\widehat{B}}v(t))$, $\widehat{U}(t):=U(x(t)+\frac{h}{2}e^{h\widehat{B}}v(t))$ and rewrite} $\frac{\widehat{E}(t_n)}{\sqrt{\widehat{U}(t_n)+C_0}}-\frac{E(\tilde{x}^{n+\frac{1}{2}})}{\sqrt{U(\tilde{x}^{n+\frac{1}{2}})+C_0}}=M_1+M_2$
		{with notations} $M_1=\frac{\widehat{E}(t_n)-E(\tilde{x}^{n+\frac{1}{2}})}{\sqrt{\widehat{U}(t_n)+C_0}}$ and
		$$M_2=E(\tilde{x}^{n+\frac{1}{2}})\frac{U(\tilde{x}^{n+\frac{1}{2}})-\widehat{U}(t_n)}{\sqrt{(\widehat{U}(t_n)+C_0)(U(\tilde{x}^{n+\frac{1}{2}})+C_0)}\left(\sqrt{\widehat{U}(t_n)+C_0}+\sqrt{U(\tilde{x}^{n+\frac{1}{2}})+C_0}\right)}.$$
		Subtracting \eqref{im-S1-SAV} from \eqref{truncation} one by one, we derive the error equations:
		\begin{equation}
			\begin{aligned}
				e_x^{n+1}=&e_x^n+he^{h\widehat{B}}e_v^n+\frac{h^2}{2}\left(\frac{\widehat{E}(t_n)}{\sqrt{\widehat{U}(t_n)+C_0}}-\frac{E(\tilde{x}^{n+\frac{1}{2}})}{\sqrt{U(\tilde{x}^{n+\frac{1}{2}})+C_0}}\right)r^{n+\frac{1}{2}}+\frac{h^2\widehat{E}(t_n)(e_r^{n+1}+e_r^n)}{4\sqrt{\widehat{U}(t_n)+C_0}}+R_x^{n+1},\\
				e_v^{n+1}=&e^{h\widehat{B}}e_v^n+h\left(\frac{\widehat{E}(t_n)}{\sqrt{\widehat{U}(t_n)+C_0}}-\frac{E(\tilde{x}^{n+\frac{1}{2}})}{\sqrt{U(\tilde{x}^{n+\frac{1}{2}})+C_0}}\right)r^{n+\frac{1}{2}}+ \frac{h}{2} \frac{\widehat{E}(t_n)(e_r^{n+1}+e_r^n)}{\sqrt{\widehat{U}(t_n)+C_0}}+R_v^{n+1},\\
				e_r^{n+1}=&e_r^n-\frac{(x^{n+1}-x^n)^\intercal}{2}\left(\frac{\widehat{E}(t_n)}{\sqrt{\widehat{U}(t_n)+C_0}}-\frac{E(\tilde{x}^{n+\frac{1}{2}})}{\sqrt{U(\tilde{x}^{n+\frac{1}{2}})+C_0}}\right)-\frac{(e_x^{n+1}-e_x^n)^\intercal\widehat{E}(t_n)}{2\sqrt{\widehat{U}(t_n)+C_0}}+R_r^{n+1}.
			\end{aligned}
		\end{equation}
		By substituting the above first equation into the third one and using  $x^{n+1}-x^n=\tilde{C}_1h\textbf{1}$ with $\textbf{1}=(1,1,1)^\intercal$, we obtain $e_r^{n+1}=e_r^n+hK_1+hK_2+(R_x^{n+1})^\intercal\frac{\widehat{E}(t_n)}{2\sqrt{\widehat{U}(t_n)+C_0}}+R_r^{n+1} $ {with $K_1=-\frac{\tilde{C}_1}{2}\textbf{1}^\intercal\left(M_1+M_2\right)$ and $K_2=-\left[e^{h\widehat{B}}e_v^n+\frac{h}{2}\left(M_1+M_2\right)r^{n+\frac{1}{2}}+\frac{h}{4}\frac{\widehat{E}(t_n)(e_r^{n+1}+e_r^n)}{\sqrt{\widehat{U}(t_n)+C_0}}\right]^\intercal\frac{\widehat{E}(t_n)}{2\sqrt{\widehat{U}(t_n)+C_0}}.$}
		{Letting} $e_w^n=\begin{pmatrix}
			e_x^n\\e_v^n\\e_r^n
		\end{pmatrix}$
		and $e_w^0=\begin{pmatrix}
			\textbf{0}_3\\\textbf{0}_3\\0
		\end{pmatrix}$, {it is obvious that
			\begin{equation}\label{component}
				\norm{e_x^n}\le\norm{e_w^n},\;\norm{e_v^n}\le\norm{e_w^n},\;\abs{e_r^n}\le\norm{e_w^n}
		\end{equation}} and
		\begin{align*}
			e_w^{n+1}=&\begin{pmatrix}
				I_3&he^{h\widehat{B}}&0\\
				O_3&e^{h\widehat{B}}&0\\
				0&0&1
			\end{pmatrix}e_w^n+
			\begin{pmatrix}
				\frac{h^2}{2}I_3&O_3&0\\
				O_3&hI_3&0\\
				0&0&h
			\end{pmatrix}
			\begin{pmatrix}
				\left(M_1+M_2\right)r^{n+\frac{1}{2}}\\
				\left(M_1+M_2\right)r^{n+\frac{1}{2}}\\
				K_1
			\end{pmatrix}\\
			&+\begin{pmatrix}
				\frac{h^2}{2}I_3&O_3&0\\
				O_3&hI_3&0\\
				0&0&h
			\end{pmatrix}
			\begin{pmatrix}
				\frac{\widehat{E}(t_n)(e_r^{n+1}+e_r^n)}{2\sqrt{\widehat{U}(t_n)+C_0}}\\
				\frac{\widehat{E}(t_n)(e_r^{n+1}+e_r^n)}{2\sqrt{\widehat{U}(t_n)+C_0}}\\
				K_2
			\end{pmatrix} +\begin{pmatrix}
				R_x^{n+1}\\
				R_v^{n+1}\\
				(R_x^{n+1})^\intercal \frac{\widehat{E}(t_n)}{2\sqrt{\widehat{U}(t_n)+C_0}}+R_r^{n+1}
			\end{pmatrix}.
		\end{align*}
		According to the assumptions of {Lemma \ref{local bounded}, Lemma \ref{global bounded} and }this Theorem, it is known that the exact solution and the numerical solution of \eqref{CPD} are both bounded. Hence we get
		\begin{align*}
			&\norm{\frac{\widehat{E}(t_n)}{\sqrt{\widehat{U}(t_n)+C_0}}}\le \frac{L}{\sqrt{C_0-c_0}}\left( \norm{x(t_n)}+\frac{h}{2}\norm{e^{h\widehat{B}}}\norm{v(t_n)}\right)\le \tilde{C}_2,\\
			&\norm{\frac{E(\tilde{x}^{n+\frac{1}{2}})}{\sqrt{U(\tilde{x}^{n+\frac{1}{2}})+C_0}}}\le \frac{L}{\sqrt{C_0-c_0}}\left( \norm{x^n}+\frac{h}{2}\norm{e^{h\widehat{B}}}\norm{v^n}\right)\le \tilde{C}_2.
		\end{align*}
		Consequently, $\norm{\frac{\widehat{E}(t_n)(e_r^{n+1}+e_r^n)}{2\sqrt{\widehat{U}(t_n)+C_0}}}\le \frac{\tilde{C}_2}{2}\left( \norm{e_w^{n+1}}+\norm{e_w^n} \right). $
		Then by using Minkowski’s inequality and Lipschitz condition, these two terms are bounded by $$\norm{M_1}\le \tilde{C}_3L\norm{e_x^n+\frac{h}{2}e^{h\widehat{B}}e_v^n}\le \tilde{C}_3L\left(\norm{e_x^n}+\frac{h}{2}\norm{e^{h\widehat{B}}}\norm{e_v^n}\right)\le \tilde{C}_3L\left( 1+h\right) \norm{e_w^n},$$
		$$\norm{M_2}\le \tilde{C}_4\norm{e_x^n+\frac{h}{2}e^{h\widehat{B}}e_v^n}\le \tilde{C}_4\left(\norm{e_x^n}+\frac{h}{2}\norm{e^{h\widehat{B}}}\norm{e_v^n}\right)\le \tilde{C}_4\left( 1+h\right) \norm{e_w^n}.$$
		In this way, we combine the two previous inequalities to get
		$$\norm{M_1+M_2}\le\norm{M_1}+\norm{M_2}\le (\tilde{C}_3L+\tilde{C}_4)(1+h)\norm{e_w^n},$$
		$$\abs{K_1}\le\frac{\tilde{C}_1}{2}\norm{\textbf{1}}\norm{M_1+M_2}\le \frac{\sqrt{3}\tilde{C}_1}{2}(\tilde{C}_3L+\tilde{C}_4)(1+h)\norm{e_w^n}.$$
		Denoting $\max\{\sqrt{3}\tilde{C}_1,\tilde{C}_2,\tilde{C}_3,\tilde{C}_4,C_{\widehat{K}}\}$ by $\tilde{C}$, thus we can estimate
		\begin{align*}
			\abs{K_2}&\le \norm{e^{h\widehat{B}}e_v^n+\frac{h}{2}\left(M_1+M_2\right)r^{n+\frac{1}{2}}+\frac{h}{4} \frac{\widehat{E}(t_n)(e_r^{n+1}+e_r^n)}{\sqrt{\widehat{U}(t_n)+C_0}} }\norm{\frac{\widehat{E}(t_n)}{2\sqrt{\widehat{U}(t_n)+C_0}}}\\
			&\le\frac{\tilde{C}}{2}\left[\norm{e^{h\widehat{B}}}\norm{e_v^n}+\frac{h}{2} (\tilde{C}+\tilde{C}L)(1+h)\norm{e_w^n}+\frac{h}{4}\tilde{C}\left(\abs{e_r^{n+1}}+\abs{e_r^n} \right) \right]  \\
			&\le\frac{\tilde{C}}{2}\left[\norm{e_w^n}+\frac{h}{2} (\tilde{C}+\tilde{C}L)(1+h)\norm{e_w^n}+\frac{h}{4}\tilde{C}\left(\norm{e_w^{n+1}}+\norm{e_w^n} \right) \right].
		\end{align*}
		On the basis of above results, it is concluded that
		\begin{equation*}
			\begin{aligned}
				\norm{e_w^{n+1}}
				\le&\norm{\begin{pmatrix}
						I_3&he^{h\widehat{B}}&0\\
						O_3&e^{h\widehat{B}}&0\\
						0&0&1
				\end{pmatrix}}\norm{e_w^n}
				+\norm{\begin{pmatrix}
						\frac{h^2}{2}I_3&O_3&0\\
						O_3&hI_3&0\\
						0&0&h
				\end{pmatrix}}
				\norm{\begin{pmatrix}
						\left(M_1+M_2\right)r^{n+\frac{1}{2}}\\
						\left(M_1+M_2\right)r^{n+\frac{1}{2}}\\
						K_1
				\end{pmatrix}}\\
				&+\norm{\begin{pmatrix}
						\frac{h^2}{2}I_3&O_3&0\\
						O_3&hI_3&0\\
						0&0&h
				\end{pmatrix}}\norm{\begin{pmatrix}
						\frac{\widehat{E}(t_n)(e_r^{n+1}+e_r^n)}{2\sqrt{\widehat{U}(t_n)+C_0}}\\
						\frac{\widehat{E}(t_n)(e_r^{n+1}+e_r^n)}{2\sqrt{\widehat{U}(t_n)+C_0}}\\
						K_2
				\end{pmatrix}}
				+\norm{\begin{pmatrix}
						R_x^{n+1}\\
						R_v^{n+1}\\
						(R_x^{n+1})^\intercal \frac{\widehat{E}(t_n)}{2\sqrt{\widehat{U}(t_n)+C_0}}+R_r^{n+1}
				\end{pmatrix}}\\
				\le&\norm{e_w^n}+\widehat{C}h(1+h)\left( \norm{e_w^{n+1}}+\norm{e_w^n}\right) +\widehat{C}h^2,
			\end{aligned}
		\end{equation*}
		in which  $\widehat{C}$ does not depend on $h$ but depends on $C_{\widehat{K}}$ and $L$. Using Gronwall's inequality and noting $\norm{e_w^0}=0$, we get
		$$\norm{e_w^n}\le Ch.$$
		As a result, we obtain the following estimations
		$$\norm{e_x^n}\le Ch,\;\norm{e_v^n}\le Ch,\;\abs{e_r^n}\le Ch,$$
		which complete the proof.
	\end{proof}

	\section{Numerical experiments}\label{sec5}
	In the previous sections, we propose a novel class of linearly implicit energy-preserving schemes (abbreviated by E2-SAV and SSAVs) for the conservative system \eqref{oso} and \eqref{CPD}. In this section, to verify our theoretical analysis results, we present the numerical performance in energy preservation, accuracy {of all schemes} and CPU time of SSAVs. At first, we introduce the global error:
	\begin{equation}\label{global error}
		error:=\frac{\norm{\xi^n-\xi(t_n)}}{\norm{\xi(t_n)}}+\frac{\norm{\eta^n-\eta(t_n)}}{\norm{\eta(t_n)}}
	\end{equation}
	and  the relative error of the energy $H(\xi,\eta,\zeta)$:
	\begin{equation}\label{ene-err}
		e_H:=\frac{\norm{H(\xi^n,\eta^n,\zeta^n)-H(\xi^0,\eta^0,\zeta^0)}}{\norm{H(\xi^0,\eta^0,\zeta^0)}},
	\end{equation}
	with $\eta:=\dot{\xi}$ and the scalar auxiliary variable $\zeta$. The reference solution is obtained by the ‘ode45’ of MATLAB and  the computation of energy is done over a long time interval with the step size $h = 0.01$. Now we carry out the following experiments to illustrate the advantages of E2-SAV and SSAVs.
	
	\subsection{E2-SAV for OSDE}\label{ne-oso}
	{To }show the superiority of our method E2-SAV, we choose the AVF (average vector field) in \cite{AVF} and the  ITO2 (implicit trapezoidal method) in \cite{06Hairer} for comparison. Since the {methods AVF and ITO2 are both} implicit, we use the standard fixed-point iteration with the error tolerance $10^{-10}$ and set the maximum number of iterations as $10^3$. {In other words}, when the error tolerance is reached or the maximum number of iterations is exceeded, the iterative will terminate. For the purpose of computing the integrals shown in AVF and ITO2, the Gauss-Legendre rules are also used.
	
	\begin{problem}\label{prob1}
		\textbf{(H\'{e}non-Heiles model)} We first consider the H\'{e}non-Heiles model, which is a classical Hamiltonian system from astronomy \cite{06Hairer,64Henon}. We adopt the form as in \cite{20Chartier,22Chartier,	23Error}:
		\begin{equation}\label{hhm}
			\frac{d}{dt}\begin{pmatrix}q\\p
			\end{pmatrix}
			=\frac{1}{\eps}J\begin{pmatrix}
				N&O_2\\
				O_2&N
			\end{pmatrix}\begin{pmatrix}
				q\\p
			\end{pmatrix}+J\nabla V(q,p)=J\nabla H_1(q,p)
		\end{equation}
		with $J = \begin{pmatrix}
			O_2&I_2\\
			-I_2&O_2
		\end{pmatrix}$, $N=\begin{pmatrix}
			1&0\\
			0&0
		\end{pmatrix}$,
		$V(q,p)=\frac{p_2^2+q_2^2}{2}+q_1^2q_2-\frac{q_2^3}{3}$ and $H_1(q,p)=\frac{p_1^2+q_1^2}{2\eps}+V(q,p)$. Denoting $u =(q_1,q_2,p_1,p_2)^\intercal$, \eqref{hhm} is exactly in the form of \eqref{IVP} with $R=\frac{1}{\eps}J\begin{pmatrix}
			N&O_2\\
			O_2&N
		\end{pmatrix}$ and $f(u)=J\nabla V(q,p)$. When $\eps$ is small, the variables $q_1, p_1$ are highly oscillatory. We take the initial value as $u_0=(0.12,0.12,0.12,0.12)^\intercal$ and the scalar auxiliary variable as $s(t)=\sqrt{V(q,p)+100}$. Applying \eqref{ESAV} to \eqref{hhm}, we get the numerical scheme E2-SAV for this system:
		\begin{equation}
			\begin{aligned}
				&u^{n+1}=\exp(hR)u^n+h\varphi(hR) Jg(\tilde{u}^{n+\frac{1}{2}},s^{n+\frac{1}{2}}),\ s^{n+1}=s^n+\frac{(u^{n+1}-u^n)^\intercal \nabla V(\tilde{u}^{n+\frac{1}{2}})}{2\sqrt{V(\tilde{u}^{n+\frac{1}{2}})+100}},
			\end{aligned}
		\end{equation}with the approximate term $\tilde{u}^{n+\frac{1}{2}}=\frac{\left( I_4+\exp(hR)\right)}{2} u^n+\frac{h}{2}\varphi(hR)Jg(u^n,s^n)$ and function $g(u,s)=\frac{\nabla V(u)}{\sqrt{V(u)+100}}s$.  Under different $\eps=1,0.1,0.01$, Figure \ref{fig11} shows the errors \eqref{ene-err} of the modified energy $\widehat{H}(q,p,s)=\frac{p_1^2+q_1^2}{2\eps}+s^2-100$ over the interval $[0,10000]$ and Figure \ref{fig12} presents the global errors \eqref{global error} until $T=1$.
	\end{problem}
	
	\begin{problem}\label{prob2}
		\textbf{(Duffing equation)} Secondly, we consider the duffing equation as follows:
		\begin{equation*}
			\frac{d}{dt}\begin{pmatrix}
				q\\p
			\end{pmatrix}=\begin{pmatrix}
				0&1\\-(\omega^2+k^2)&0
			\end{pmatrix}\begin{pmatrix}
				q\\p
			\end{pmatrix}+\begin{pmatrix}
				0\\2k^2q^3
			\end{pmatrix},\quad \begin{pmatrix}
				q(0)\\p(0)
			\end{pmatrix}=\begin{pmatrix}0\\\omega
			\end{pmatrix}
		\end{equation*}
		with its Hamiltonian
		\begin{equation*}
			H(q,p)=\frac{1}{2}p^2+\frac{1}{2}(\omega^2+k^2)q^2-\frac{k^2}{2}q^4.
		\end{equation*}
		The exact solution of this system is $q(t)=sn(\omega t;k/\omega)$ with the Jacobi elliptic function $sn$. {Letting  the scalar auxiliary variable be} $s(t)=\sqrt{-\frac{k^2}{2}q^4+100}$,  we can apply \eqref{q^n}-\eqref{s^n} directly to this system. Under the cases where $\omega=5,10,20$ and $k=0.07$, Figure \ref{fig21} shows the errors \eqref{ene-err} of the modified energy $\widehat{H}(q,p,s)=\frac{1}{2}p^2+\frac{1}{2}(\omega^2+k^2)q^2+s^2-100$ over the interval $[0,10000]$ and Figure \ref{fig22} displays the global errors \eqref{global error} until $T=1$.
	\end{problem}
	
	\begin{problem}\label{prob3}
		\textbf{(sine-Gordon equation)} This test is devoted the sine-Gordon equation with periodic boundary conditions \cite{06Franco}
		\begin{align*}
			\begin{array}
				[c]{ll}%
				\dfrac{\partial^{2}u}{\partial t^{2}}=\dfrac{\partial^{2}u}{\partial
					x^{2}}-\sin u,\ \ \ -1<x<1,\ \ t>0, \ \ \
				u(-1,t)=u(1,t). &
			\end{array}
		\end{align*}
		With second-order symmetric differences on the spatial variable, the above PDE can be transformed into the following   second-order ODEs:
		$$\frac{d^{2}U}{dt^{2}}+QU=F(U),\quad 0<t\leq t_{end},$$
		where $U(t)=(u_{1}(t),\ldots,u_{N}(t))^{T}$, $F(U)   =-\sin(U)=-\big(\sin u_{1},\ldots,\sin u_{N}\big)^{T}$  with $u_{i}(t)\approx u(x_{i},t)$ for $i=1,2,\ldots,N$, and $$Q=\dfrac{1}{\Delta x^{2}}\left(
		\begin{array}
			[c]{ccccc}%
			2 & -1 &  &  & -1\\
			-1 & 2 & -1 &  & \\
			& \ddots & \ddots & \ddots & \\
			&  & -1 & 2 & -1\\
			-1 &  &  & -1 & 2
		\end{array}\right)\qquad \textmd{with}\quad \Delta x=2/N.$$
		Following the paper \cite{06Franco},  the initial
		conditions are chosen as
		\[
		U(0)=(\pi)_{i=1}^{N},\ \ \ U_{t}(0)=\sqrt{N}\Big(0.01+\sin(\dfrac{2 \pi i}{N})\Big)_{i=1}^{N},
		\]
		and the Hamiltonian is $H(U,U_t)=\frac{1}{2}(U_t)^\intercal U_t+\frac{1}{2}U^\intercal MU+V(U)$ with $V(U)=-(\cos u_1+\cos u_2+\dots +\cos u_N)$. Introducing the scalar $s(t)=\sqrt{V(U)+100}$, we can apply the scheme E2-SAV to this problem. Taking $N=16,32,64$, Figure \ref{fig31} shows the errors \eqref{ene-err} of the modified energy $\widehat{H}(U,U_t,s)=\frac{1}{2}(U_t)^\intercal U_t+\frac{1}{2}U^\intercal QU+s^2-100$ on the interval $[0,2000]$ and Figure \ref{fig32} displays the global errors \eqref{global error} until $T=1$.
	\end{problem}

	From the results in the three tests, it can be observed that the AVF and ITO2 {do not have modified energy-preserving property but E2-SAV exactly conserves the modified energy. Meanwhile,  the E2-SAV shows a second order  accuracy, and} when the system is highly oscillatory, its accuracy is almost unaltered in comparison with AVF and ITO2.
	
	\begin{figure}[t!]
		\centering
		\begin{tabular}[c]{ccc}%
			\subfigure{\includegraphics[width=4.7cm,height=4.2cm]{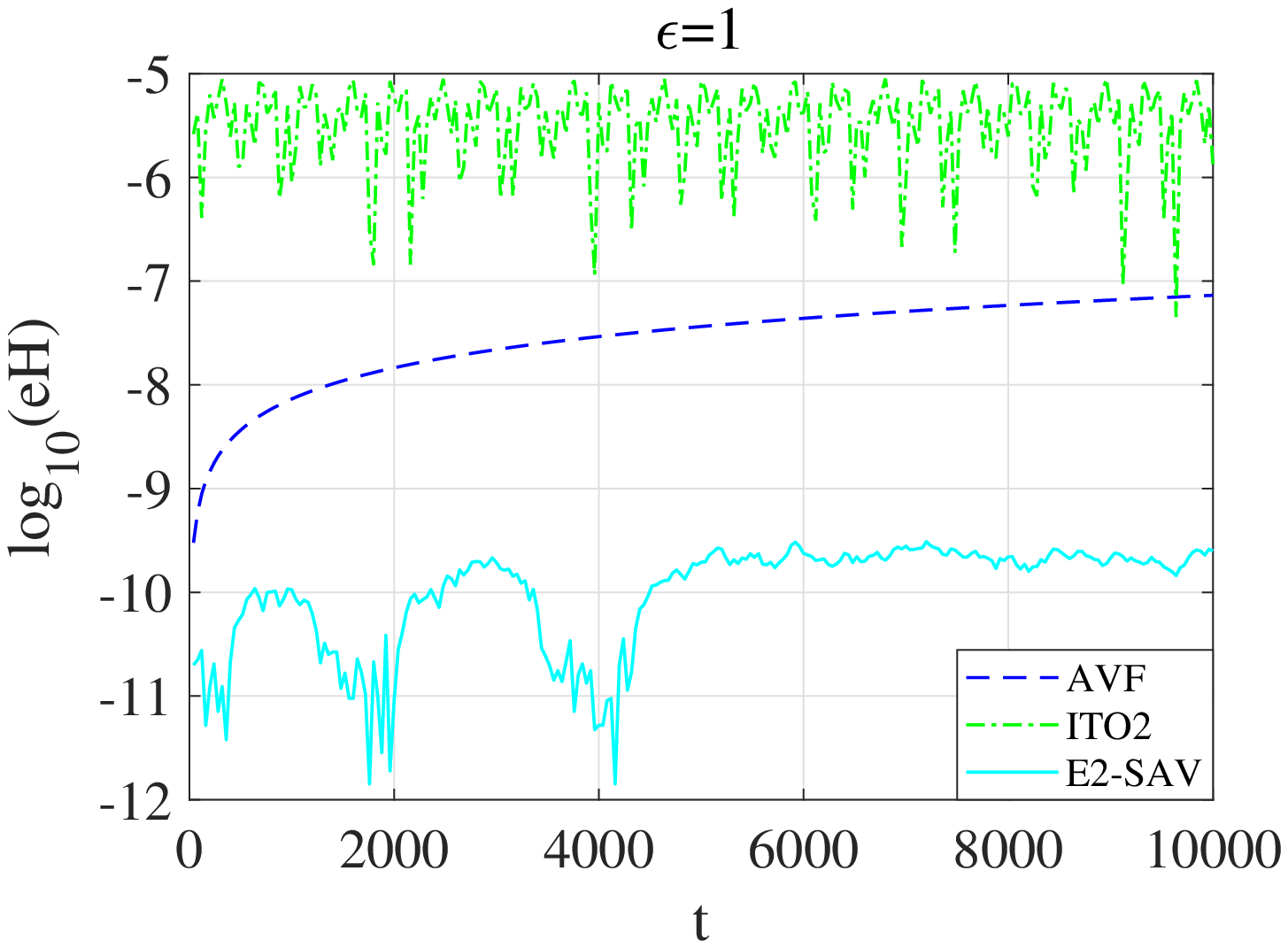}}			\subfigure{\includegraphics[width=4.7cm,height=4.2cm]{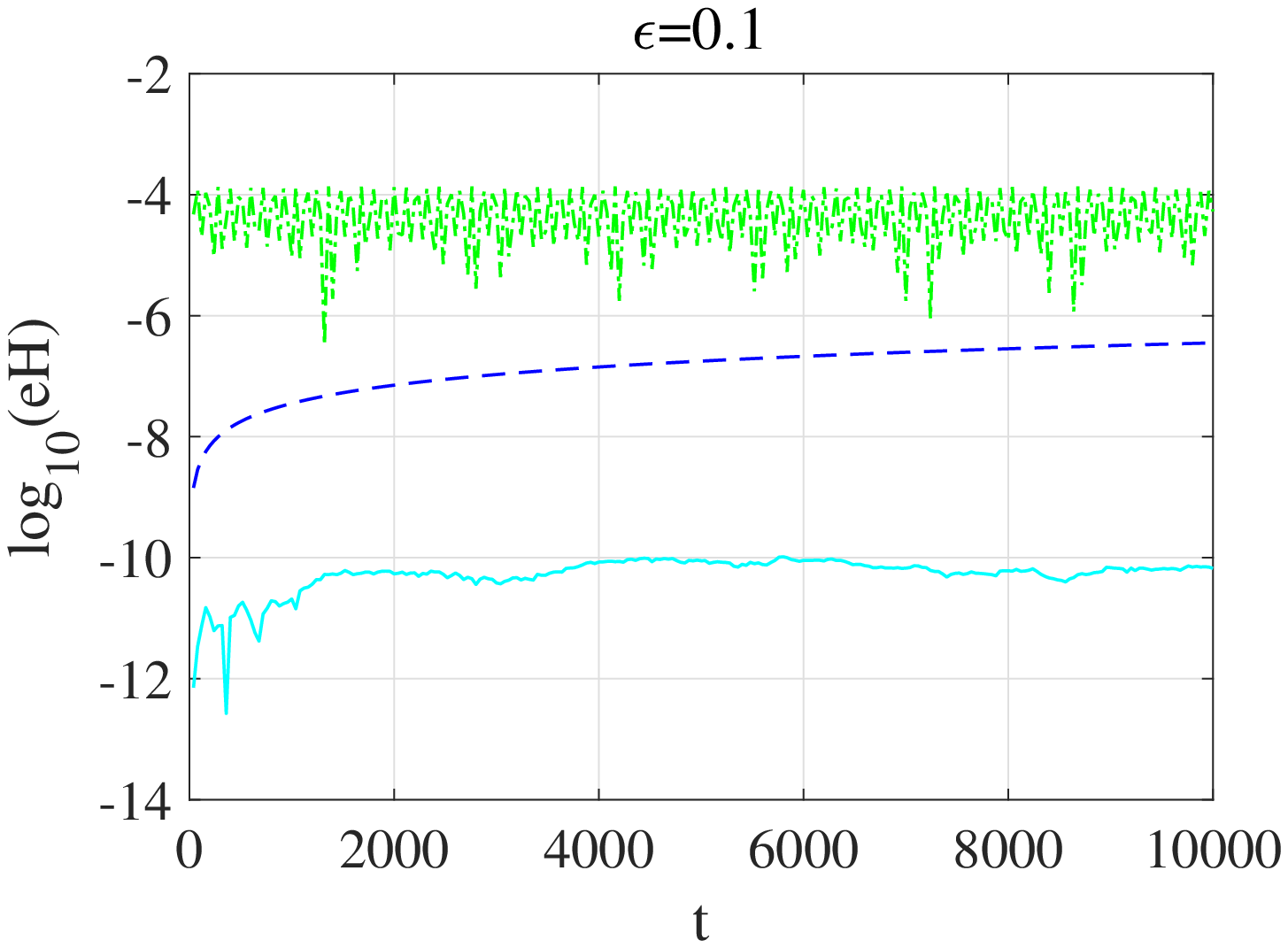}}
			\subfigure{\includegraphics[width=4.7cm,height=4.2cm]{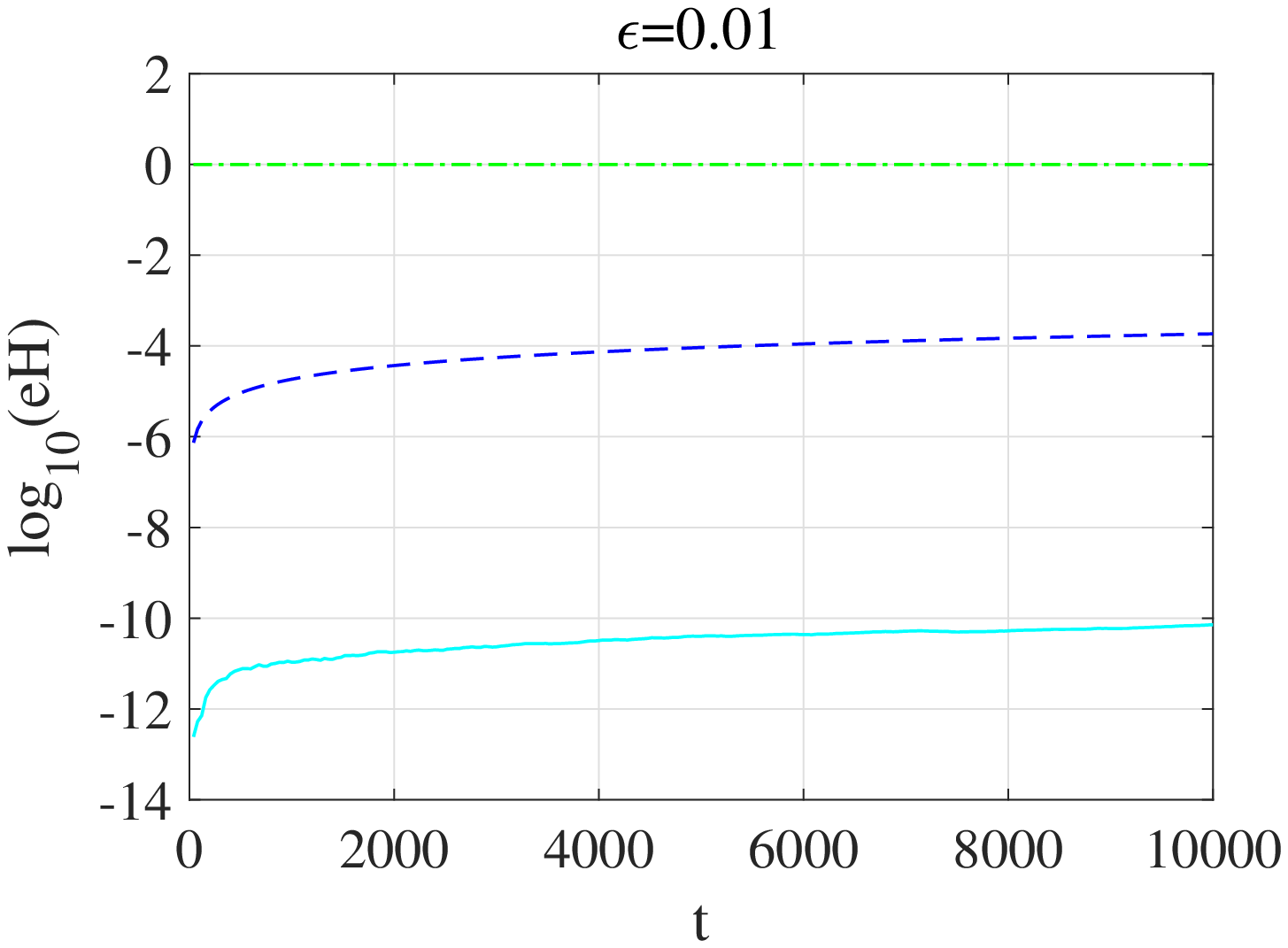}}		
		\end{tabular}
		\caption{{Problem \ref{prob1}.} Evolution of the error (\ref{ene-err}) with modified energy $\widehat{H}(q,p,s)=\frac{p_1^2+q_1^2}{2\eps}+s^2-100$ as function of time $t_n = nh$ under different $\eps$.}\label{fig11}
	\end{figure}
	
	\begin{figure}[t!]
		\centering
		\begin{tabular}[c]{ccc}%
			\subfigure{\includegraphics[width=4.7cm,height=4.2cm]{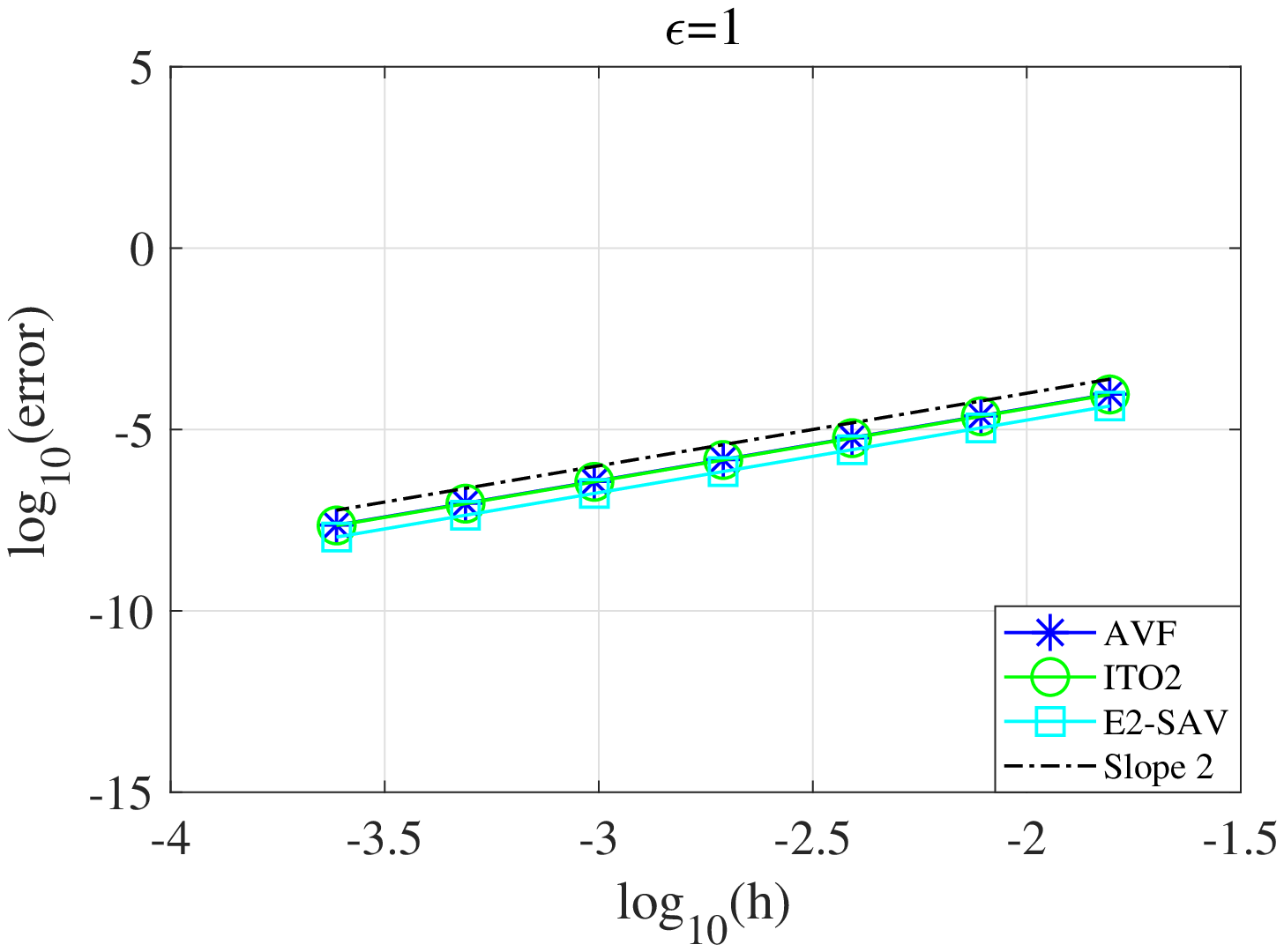}}			\subfigure{\includegraphics[width=4.7cm,height=4.2cm]{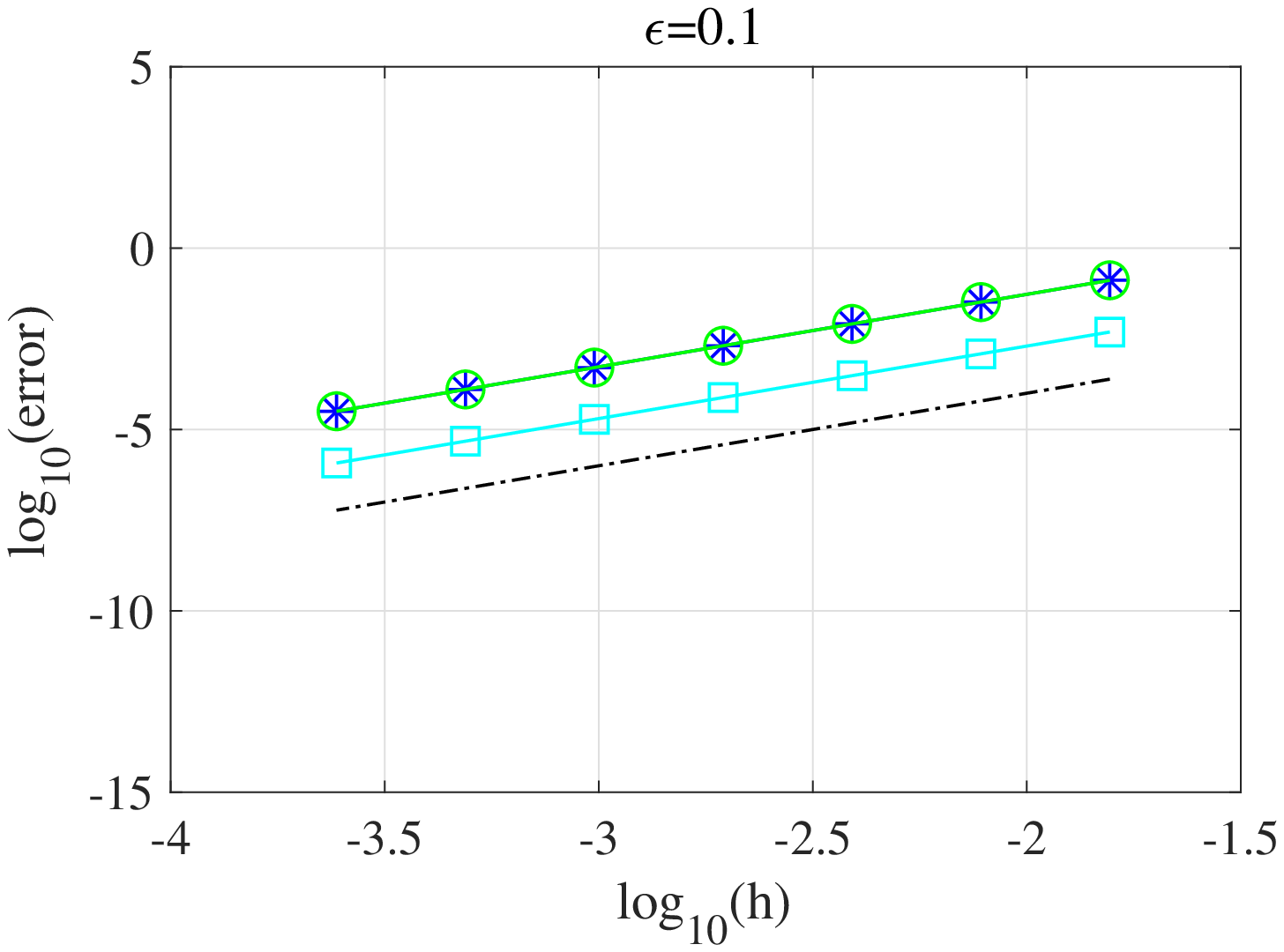}}
			\subfigure{\includegraphics[width=4.7cm,height=4.2cm]{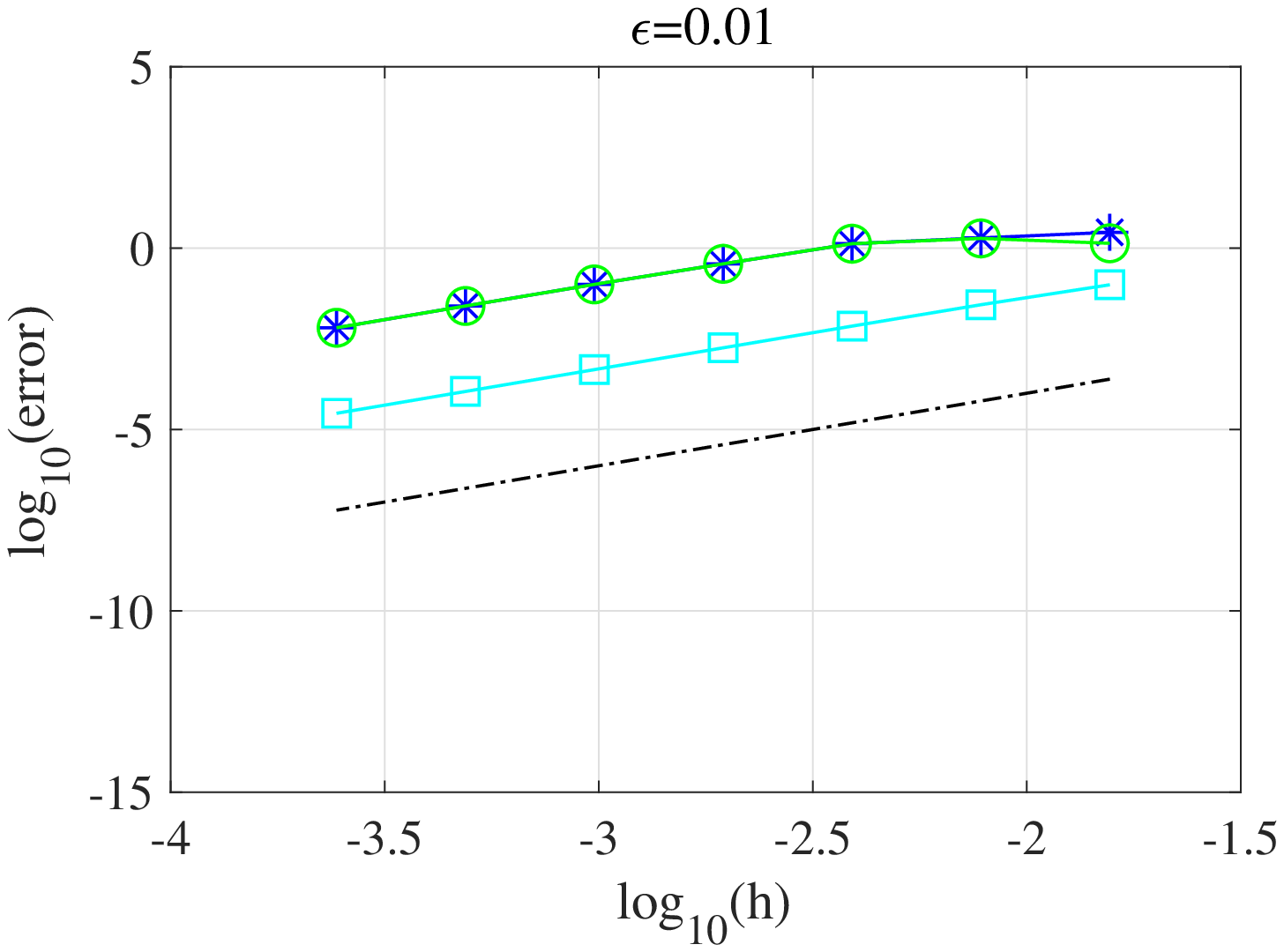}}		
		\end{tabular}
		\caption{Problem \ref{prob1}. The global errors \eqref{global error} with $T = 1$ and $h = 1/2^k$ for $k = 6,7,\dots,12$ under different $\eps$.}\label{fig12}
	\end{figure}
	
	\begin{figure}[t!]
		\centering
		\begin{tabular}[c]{ccc}%
			\subfigure{\includegraphics[width=4.7cm,height=4.2cm]{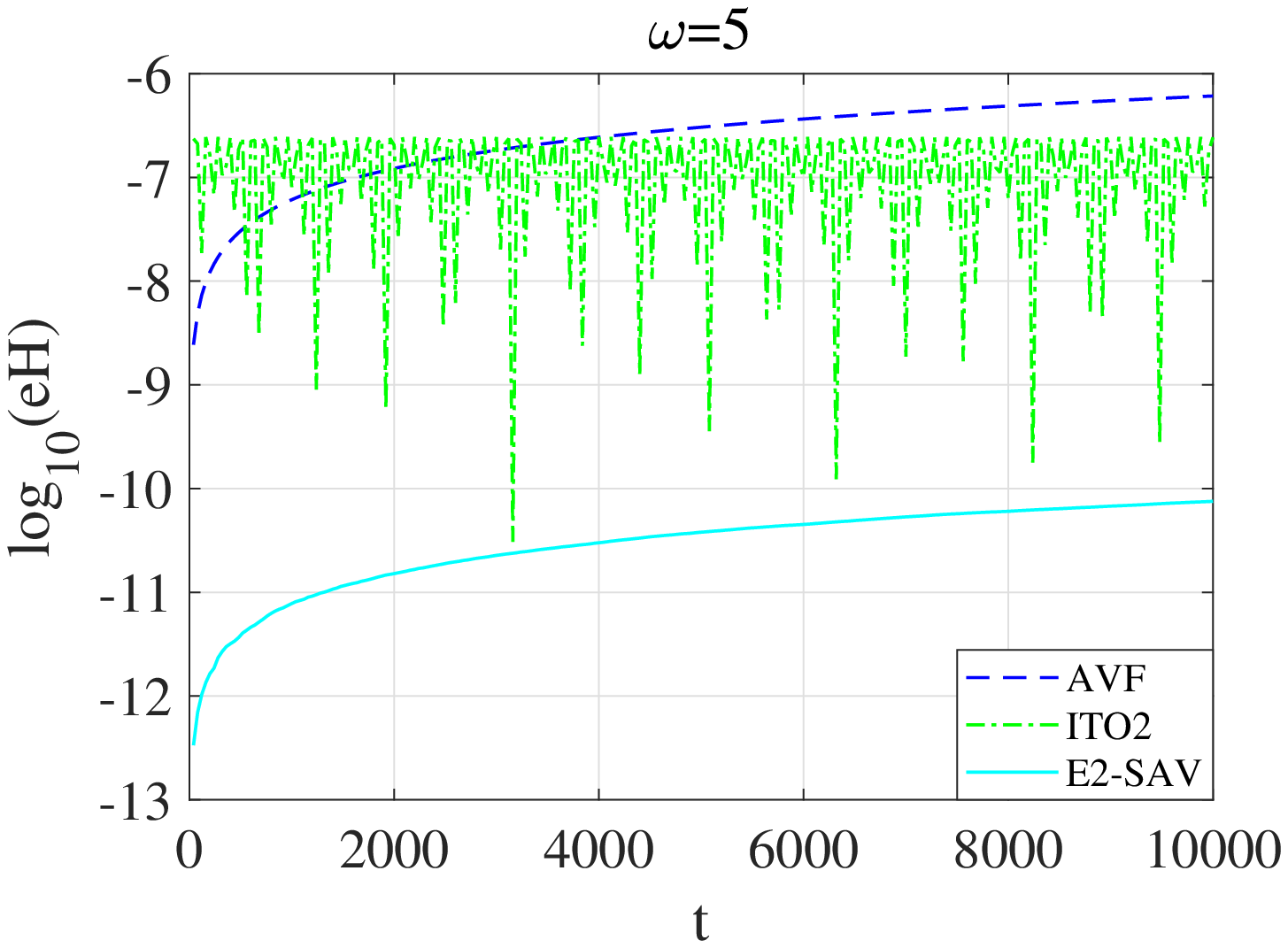}}			\subfigure{\includegraphics[width=4.7cm,height=4.2cm]{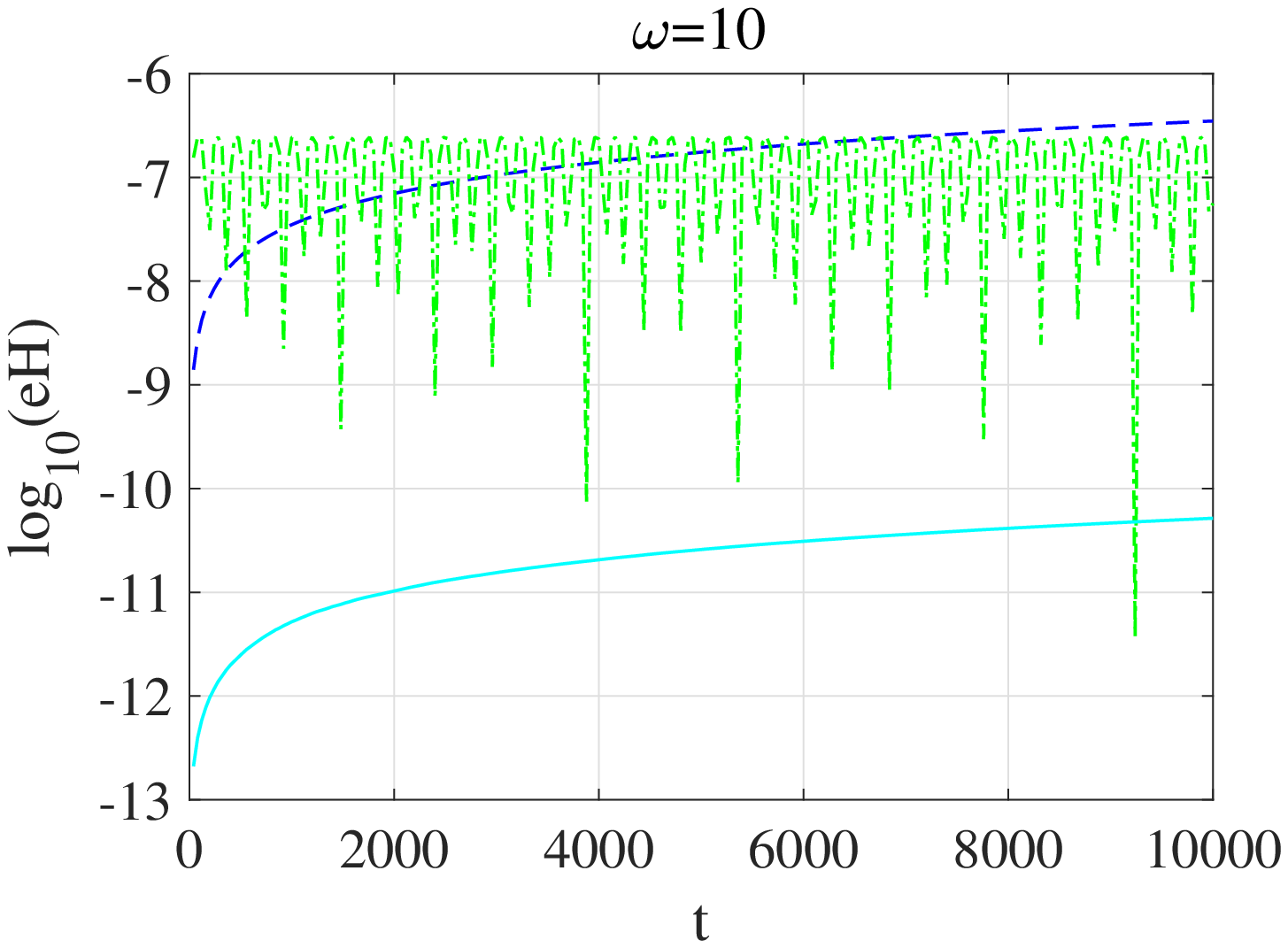}}
			\subfigure{\includegraphics[width=4.7cm,height=4.2cm]{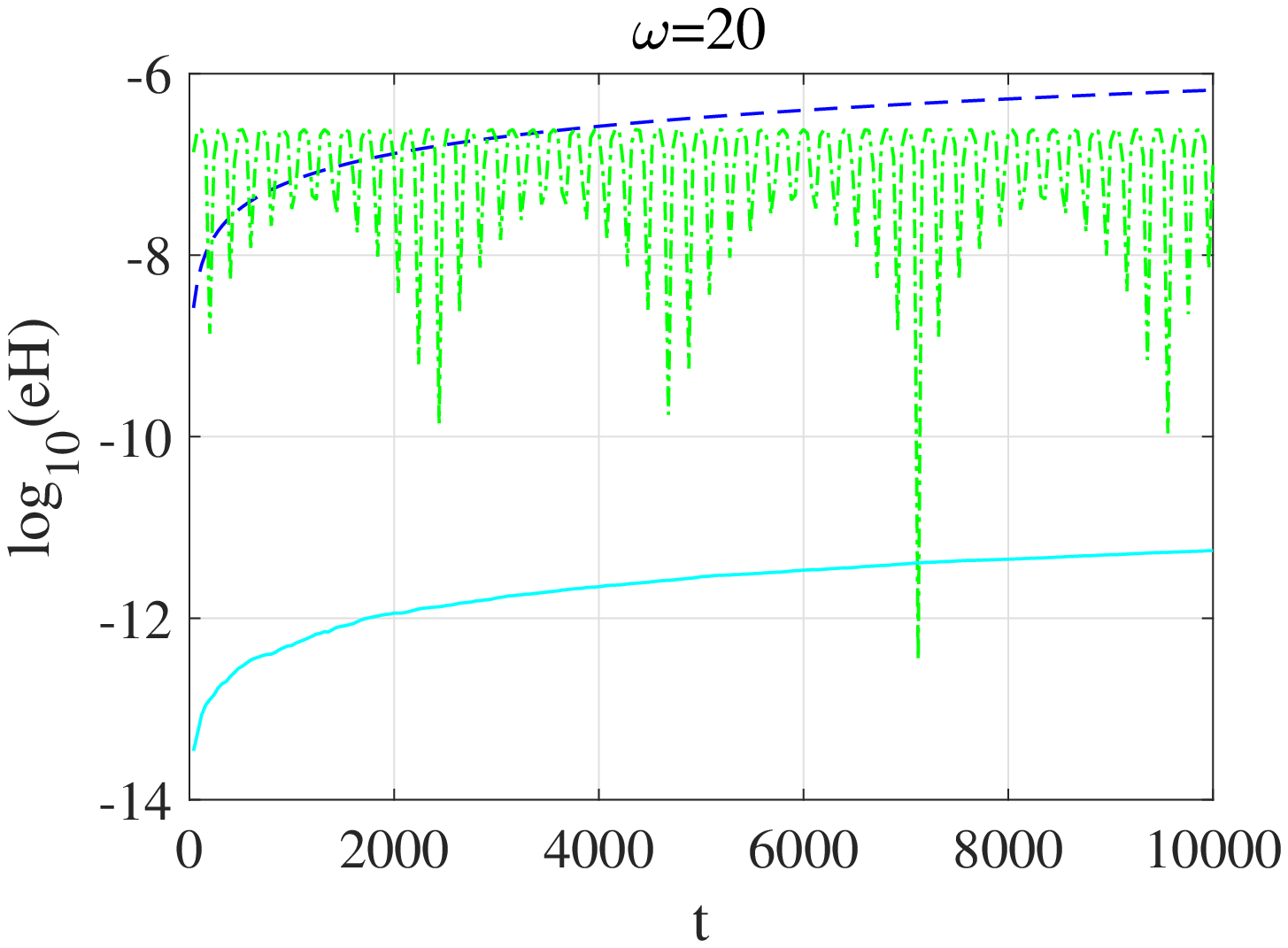}}		
		\end{tabular}
		\caption{{Problem \ref{prob2}.} Evolution of the error (\ref{ene-err}) with modified energy $\widehat{H}(q,p,s)=\frac{1}{2}p^2+\frac{1}{2}(\omega^2+k^2)q^2+s^2-100$ as function of time $t_n = nh$ under different $\omega$.}\label{fig21}
	\end{figure}
	
	\begin{figure}[t!]
		\centering
		\begin{tabular}[c]{ccc}%
			\subfigure{\includegraphics[width=4.7cm,height=4.2cm]{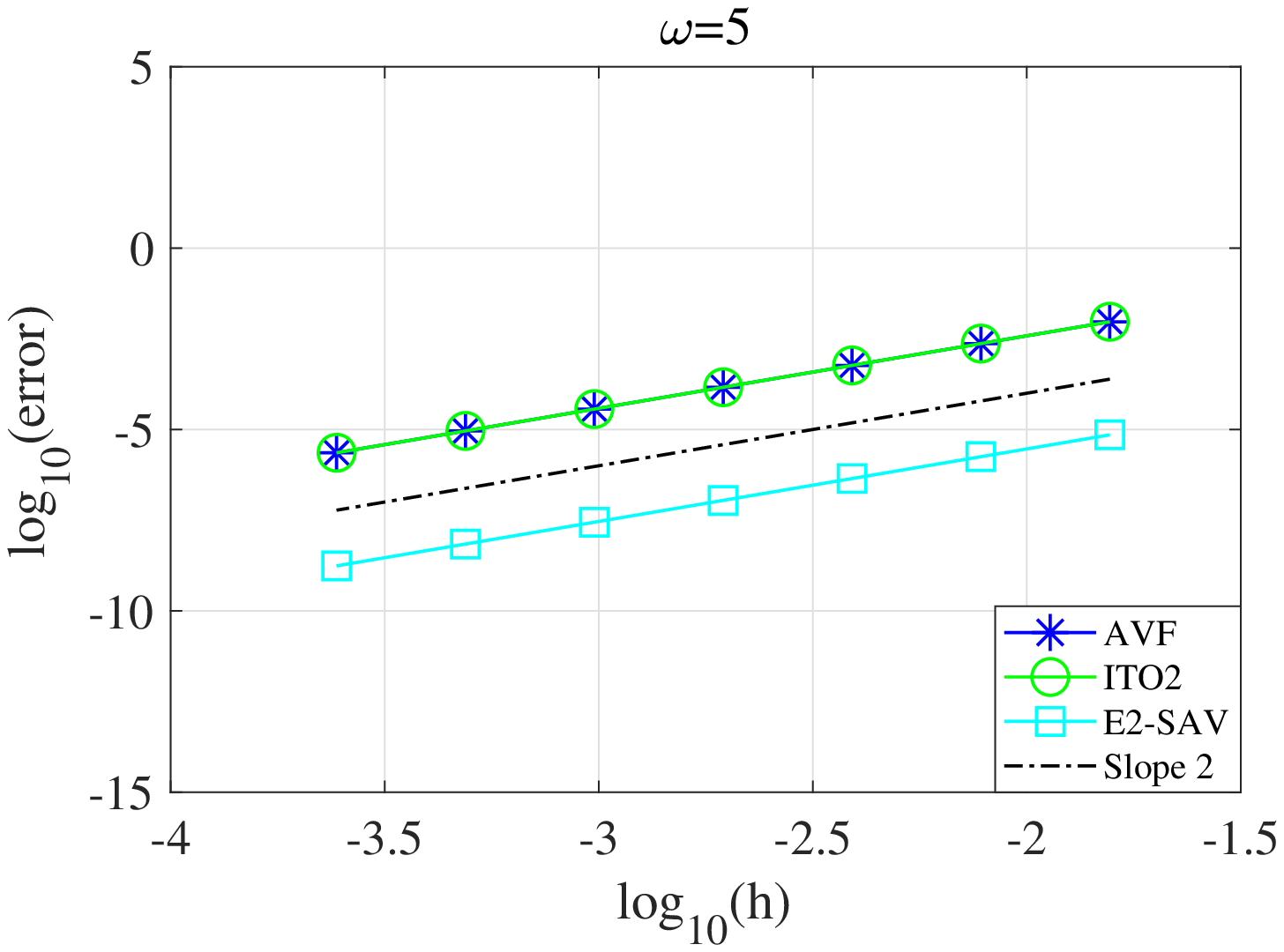}}			\subfigure{\includegraphics[width=4.7cm,height=4.2cm]{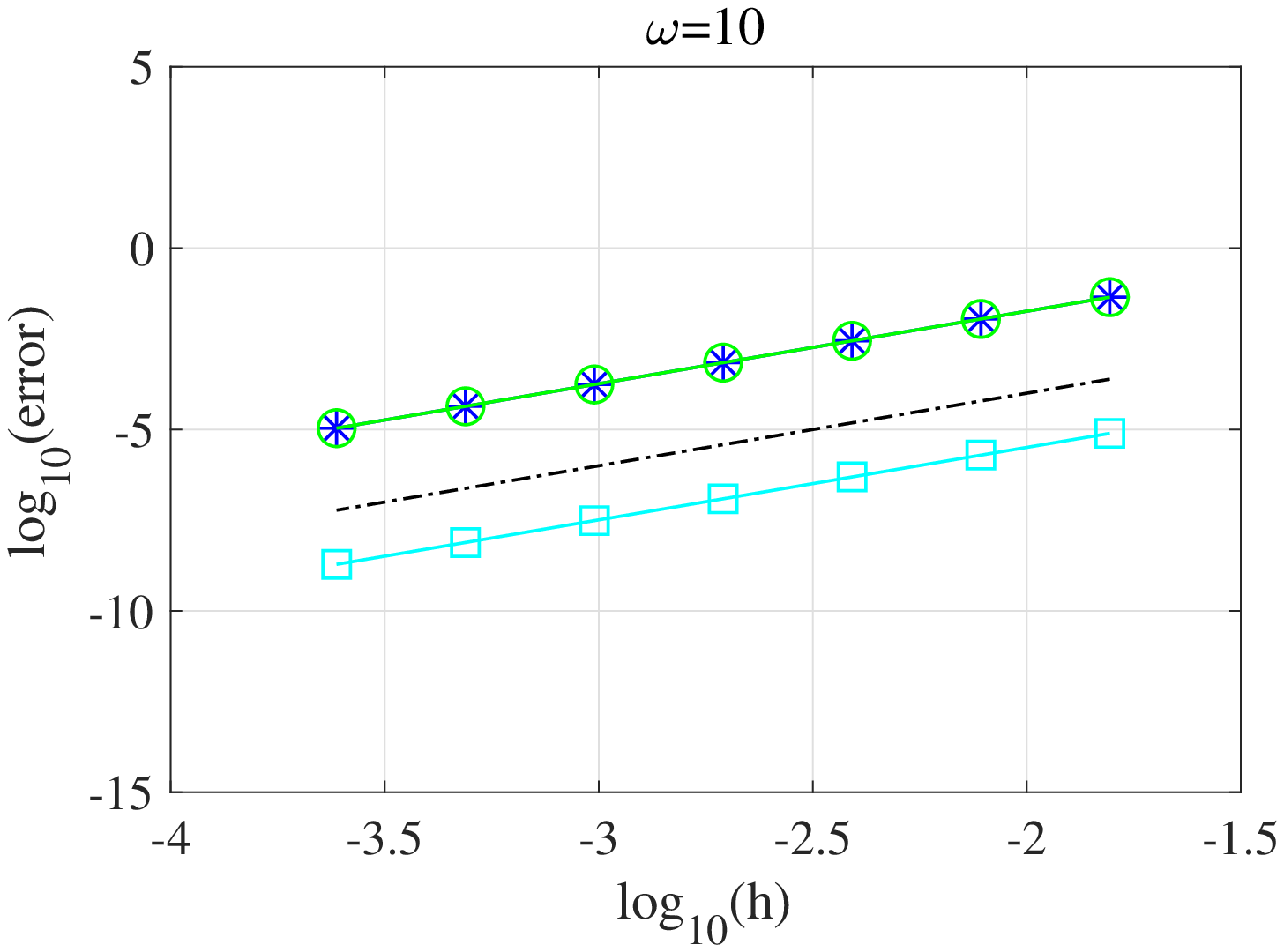}}
			\subfigure{\includegraphics[width=4.7cm,height=4.2cm]{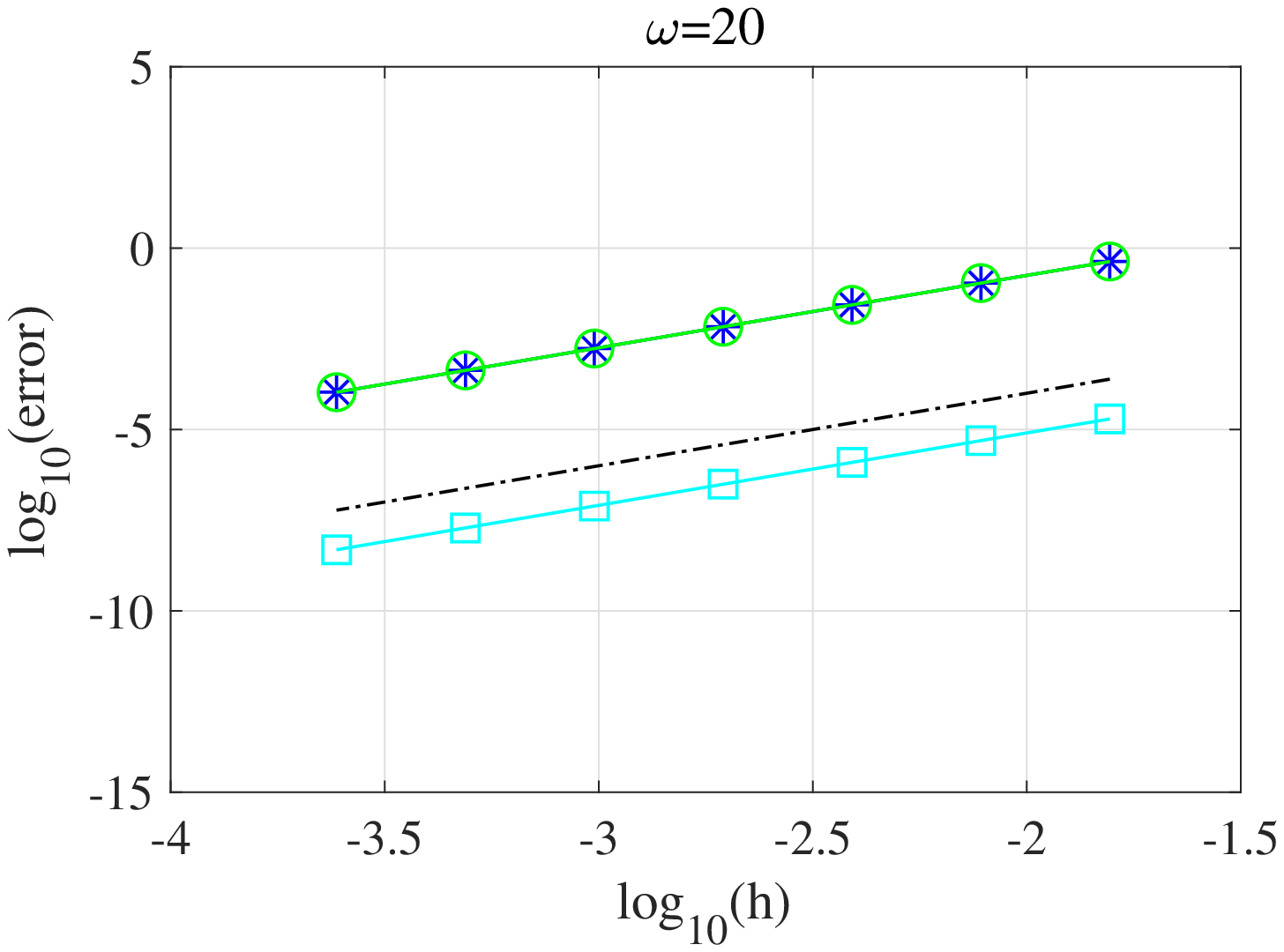}}		
		\end{tabular}
		\caption{{Problem \ref{prob2}.} The global errors \eqref{global error} with $T = 1$ and $h = 1/2^k$ for $k = 6,7,\dots,12$ under different $\omega$.}\label{fig22}
	\end{figure}
	
	\begin{figure}[t!]
		\centering
		\begin{tabular}[c]{ccc}%
			\subfigure{\includegraphics[width=4.7cm,height=4.2cm]{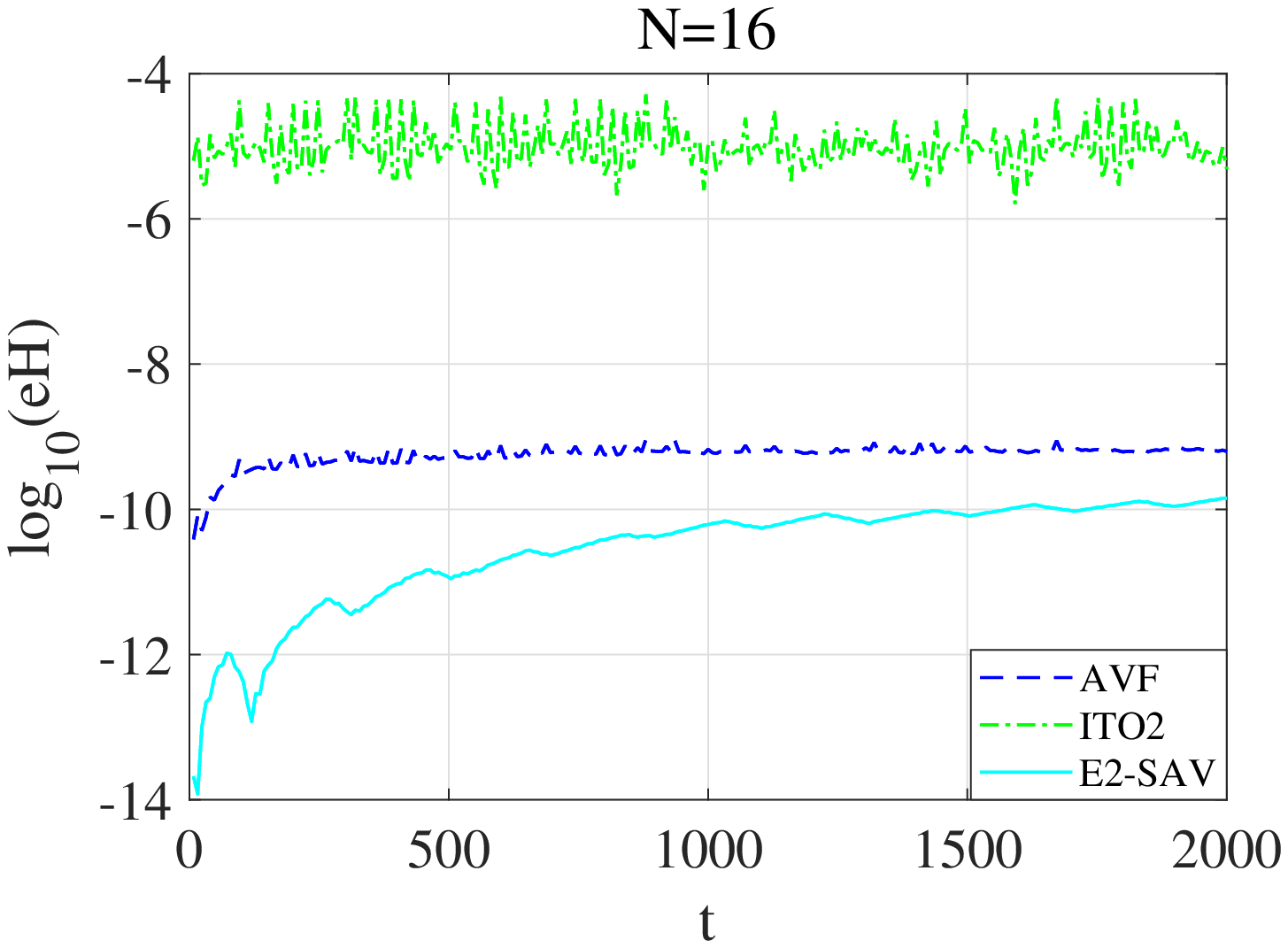}}			\subfigure{\includegraphics[width=4.7cm,height=4.2cm]{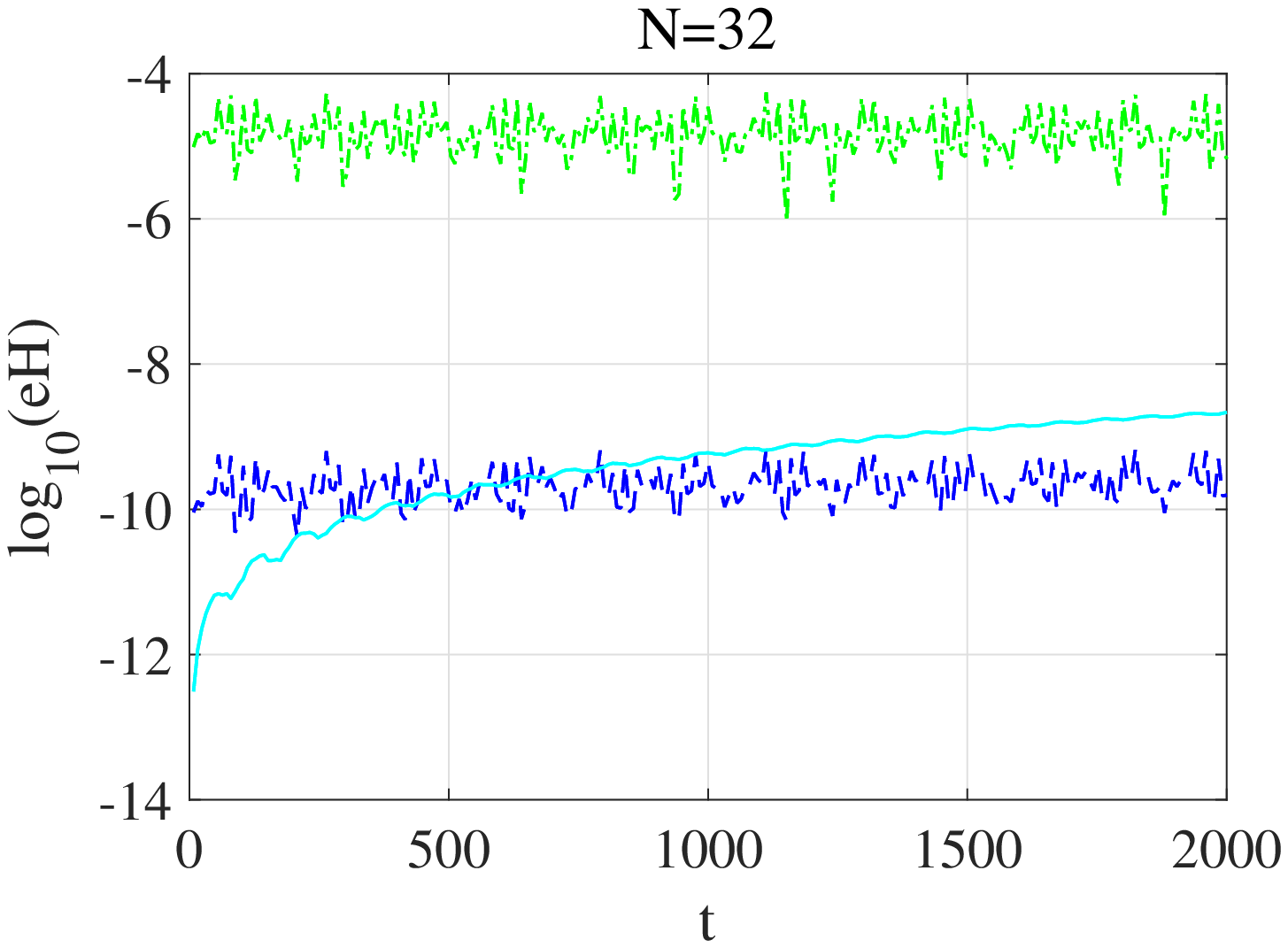}}
			\subfigure{\includegraphics[width=4.7cm,height=4.2cm]{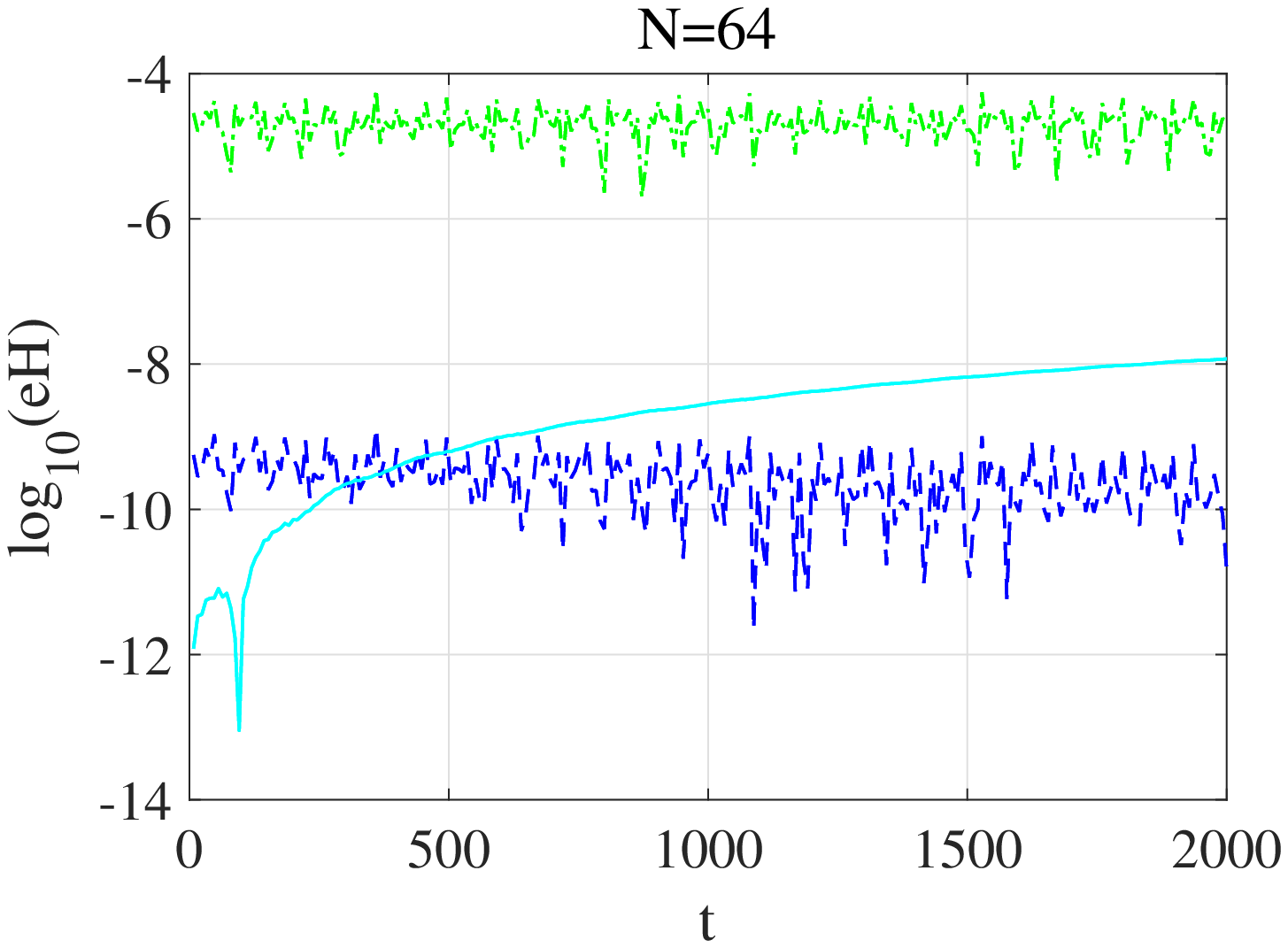}}		
		\end{tabular}
		\caption{Problem \ref{prob3}. Evolution of the error (\ref{ene-err}) with modified energy $\widehat{H}(U,U_t,s)=\frac{1}{2}(U_t)^\intercal U_t+\frac{1}{2}U^\intercal QU+s^2-100$ as function of time $t_n = nh$ under different $N$.}\label{fig31}
	\end{figure}
	
	\begin{figure}[t!]
		\centering
		\begin{tabular}[c]{ccc}%
			\subfigure{\includegraphics[width=4.7cm,height=4.2cm]{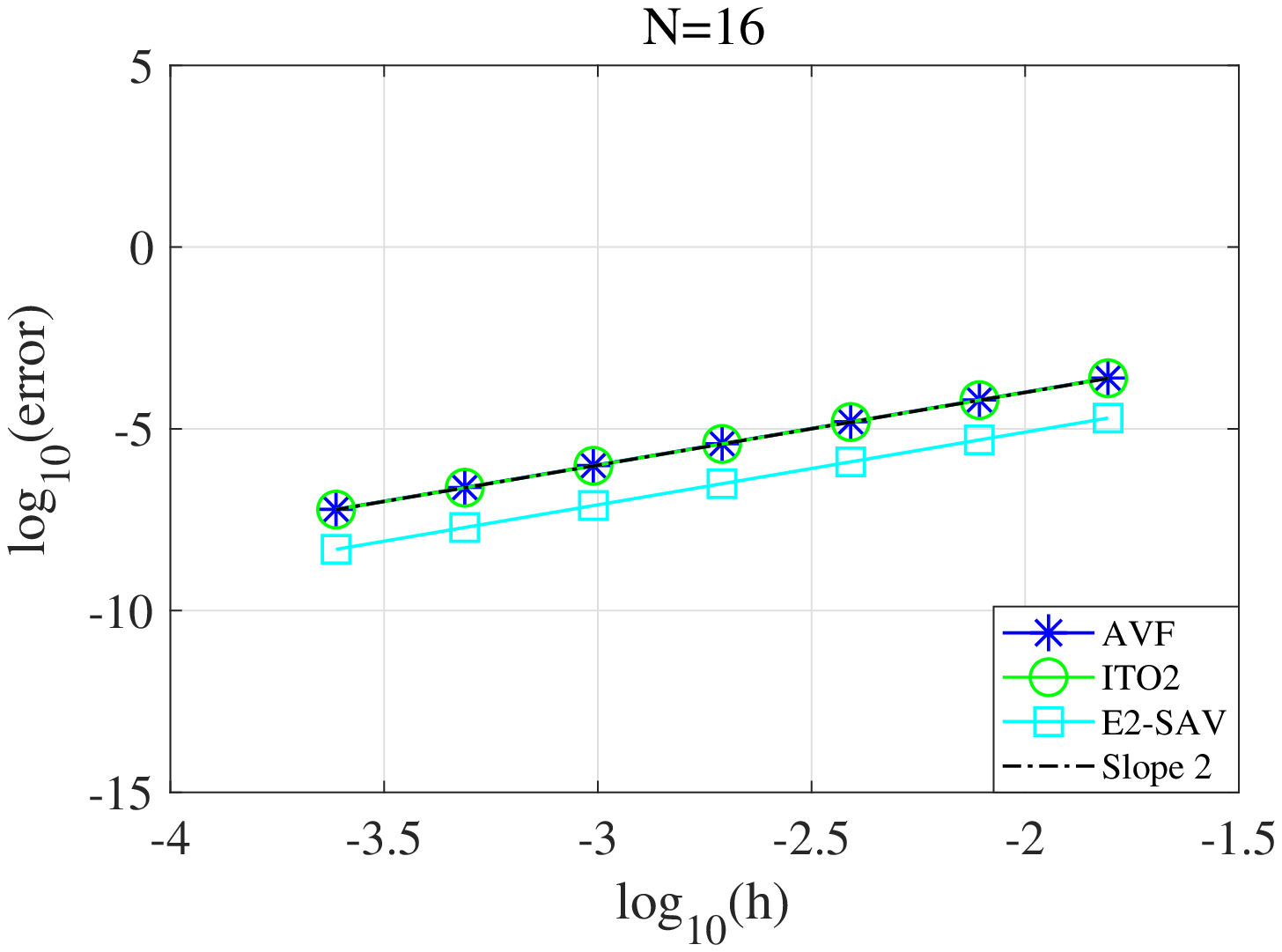}}			\subfigure{\includegraphics[width=4.7cm,height=4.2cm]{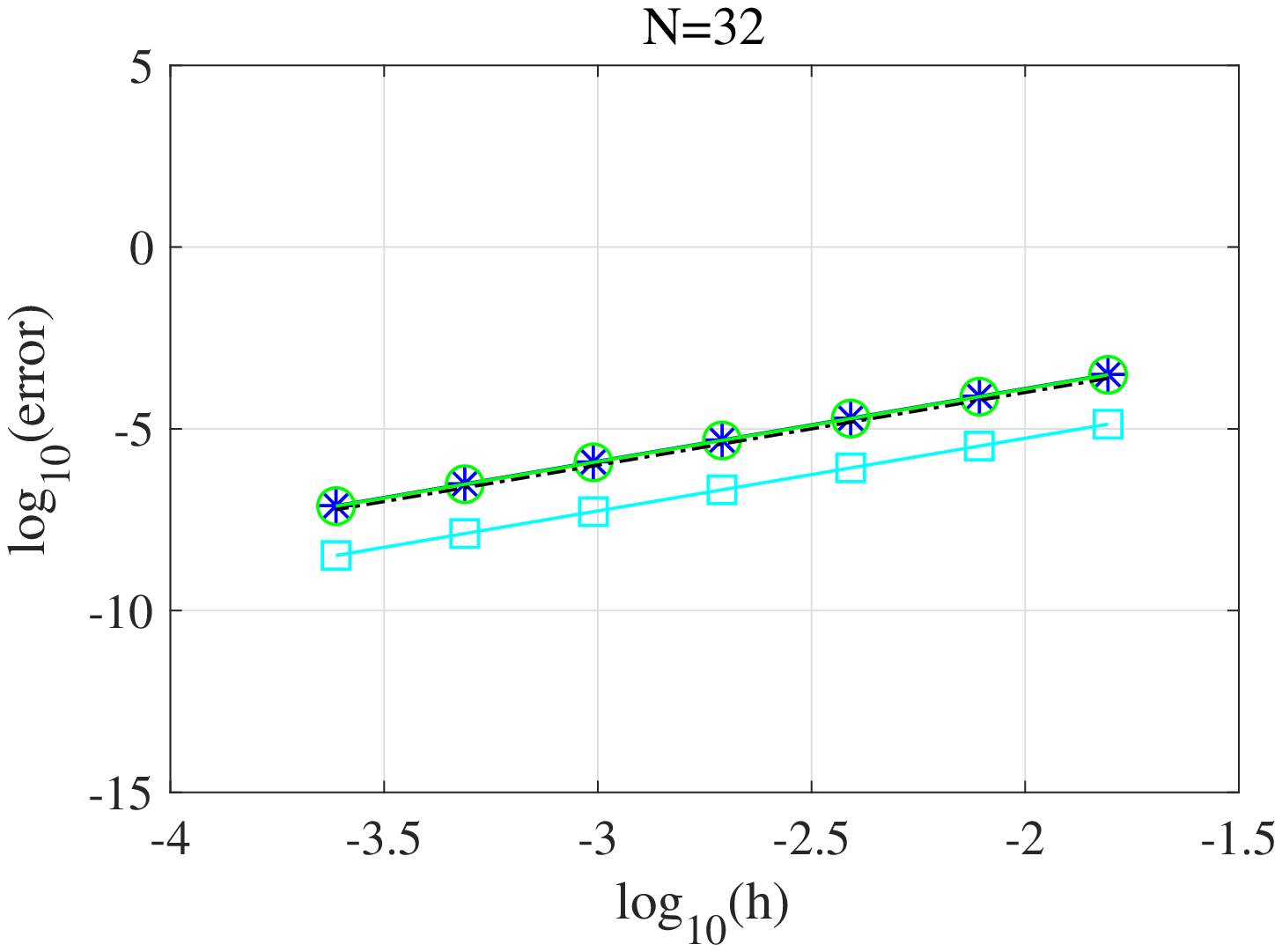}}
			\subfigure{\includegraphics[width=4.7cm,height=4.2cm]{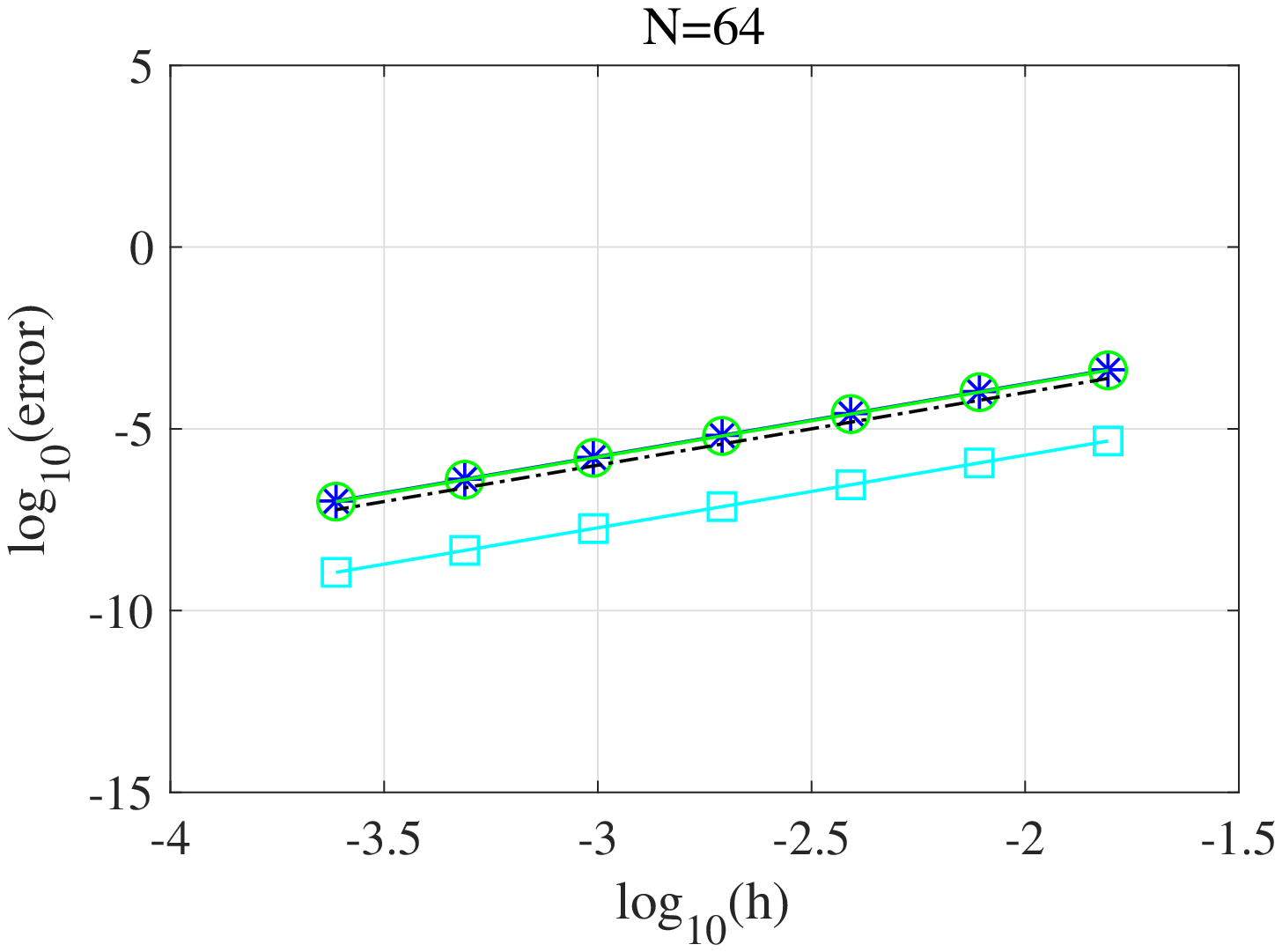}}		
		\end{tabular}
		\caption{Problem \ref{prob3}. The global errors \eqref{global error} with $T = 1$ and $h = 1/2^k$ for $k = 6,7,\dots,12$ under different $N$.}\label{fig32}
	\end{figure}
	\subsection{SSAVs for the CPD}	
	{In this part}, we choose Boris algorithm (single step) given {in \cite{70Boris}} {and AVF} for comparison  in the energy conservation, accuracy, and CPU time. We deal with the implicit scheme AVF in the same way as subsection \ref{ne-oso}.
	
	\begin{problem}\label{prob4}
		\textbf{(CPD with a constant magnetic field)} We consider the case where the magnetic field is constant $B(x)=\frac{1}{\eps}(0,0,1)^\intercal$ and $U(x)=\frac{1}{100\sqrt{x_1^2+x_2^2}}$ with $C_0=1$. For initial values we take $x(0) = (0.7, 1, 0.1)^\intercal$ and $v(0) = (0.9, 0.5, 0.4)^\intercal$. Figure \ref{fig31} shows the modified energy conservation of the obtained methods under different $\eps$. To test the accuracy of them, we numerically solve the CPD until $T = 1$, and the global errors \eqref{global error} are presented in Figure \ref{fig32}.
	\end{problem}
	
	\begin{problem}\label{prob5}
		\textbf{(CPD with a general magnetic field)} Last but not least, we choose the general magnetic field as
		$$B(x)=\nabla\times\frac{1}{3\eps}(-x_2\sqrt{x_1^2+x_2^2},-x_1\sqrt{x_1^2+x_2^2},0)^\intercal={\frac{1}{\eps}}(0,0,\sqrt{x_1^2+x_2^2})^\intercal.$$
		The scalar potential and the initial values are the same as them in Problem \ref{prob4}. The modified energy conservation of SSAVs are shown in Figure \ref{fig51}. The global errors \eqref{global error} with $T=1$ are presented in Figure \ref{fig52}.
		In addition, the CPU time are displayed in Figure \ref{fig53}.
	\end{problem}
	
	{From the  numerical results shown in Figures \ref{fig41}-\ref{fig53}}, we can draw the following conclusions
	\begin{enumerate}
		\item In terms of energy conservation, we can observe a significant difference between these six methods. All the  SSAVs  hold a long exact energy-preserving behaviour but the Boris algorithm and AVF  do not have such conservation.
		\item From Figure \ref{fig42} and Figure \ref{fig52}, we can observe that our methods have better accuracy than Boris algorithm and AVF method especially when $\eps$ is small. In particular, we can {notice that} the global error lines of our methods S1-SAV, S2-SAV, S4-SAV and S6-SAV are respectively nearly parallel to the lines of slope 1, 2, 4 and 6,  indicating that they are of order 1, 2, 4 and 6 respectively.
		\item For CPU time in Figure \ref{fig53}, obviouly, the cost of AVF is more expensive than SSAVs. It must also be mentioned that the tolerance was not met until $10^3$ iterations for AVF  in many cases. This fact further demonstrates the tremendous computational efficiency of our methods compared with the AVF method.
	\end{enumerate}
	
	\begin{figure}[t!]
		\centering
		\begin{tabular}[c]{ccc}%
			\subfigure{\includegraphics[width=4.7cm,height=4.2cm]{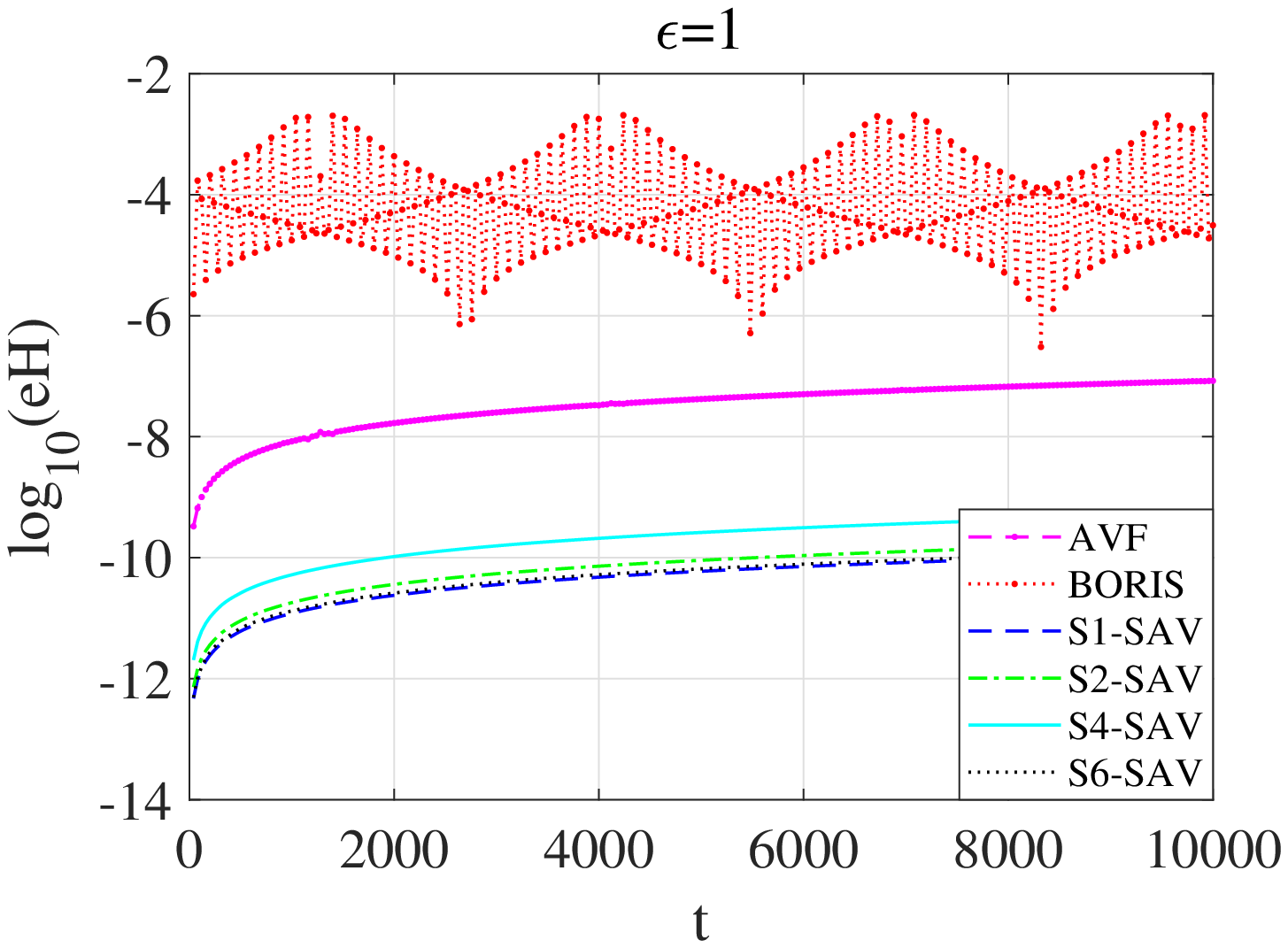}}			\subfigure{\includegraphics[width=4.7cm,height=4.2cm]{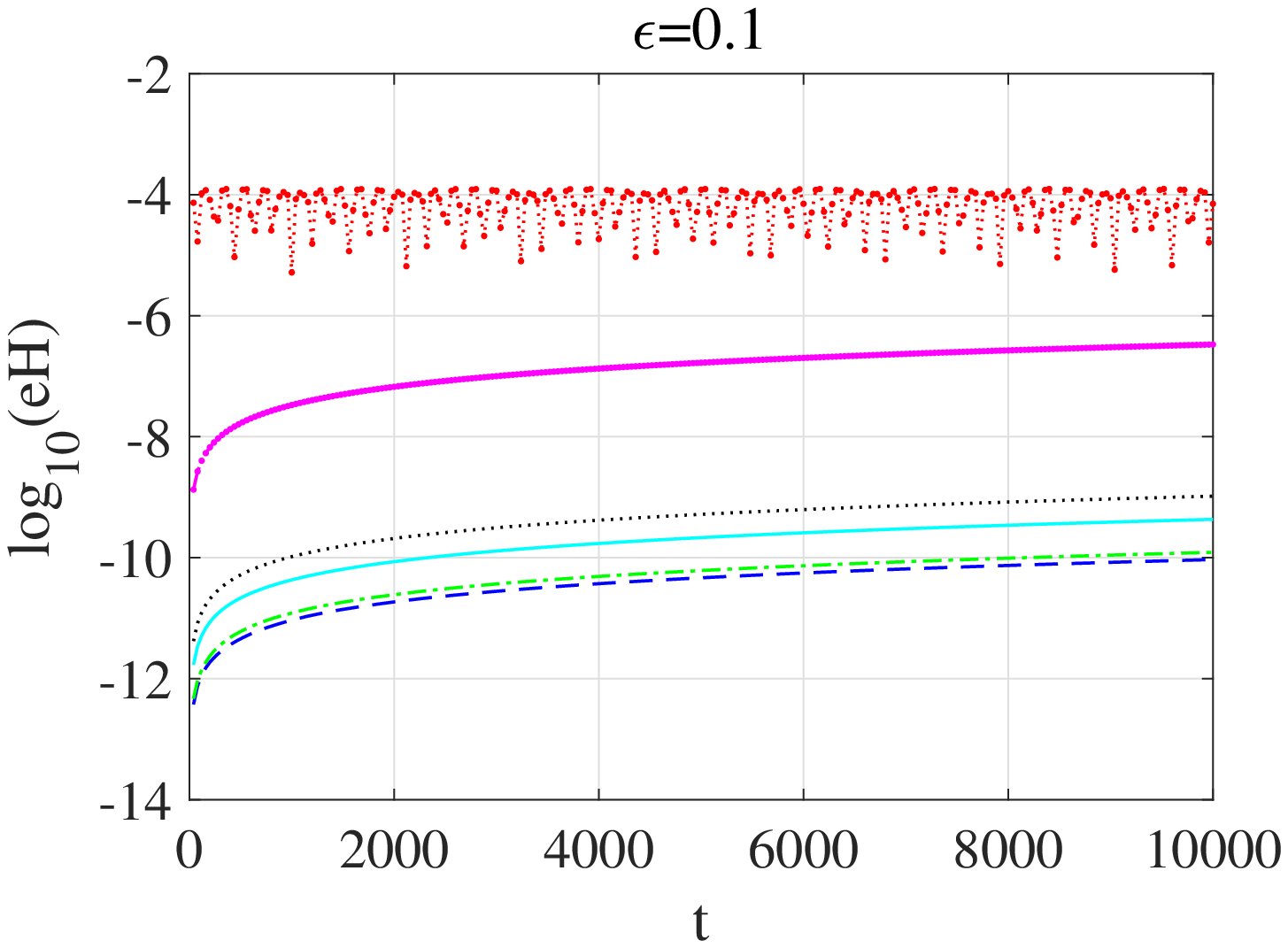}}
			\subfigure{\includegraphics[width=4.7cm,height=4.2cm]{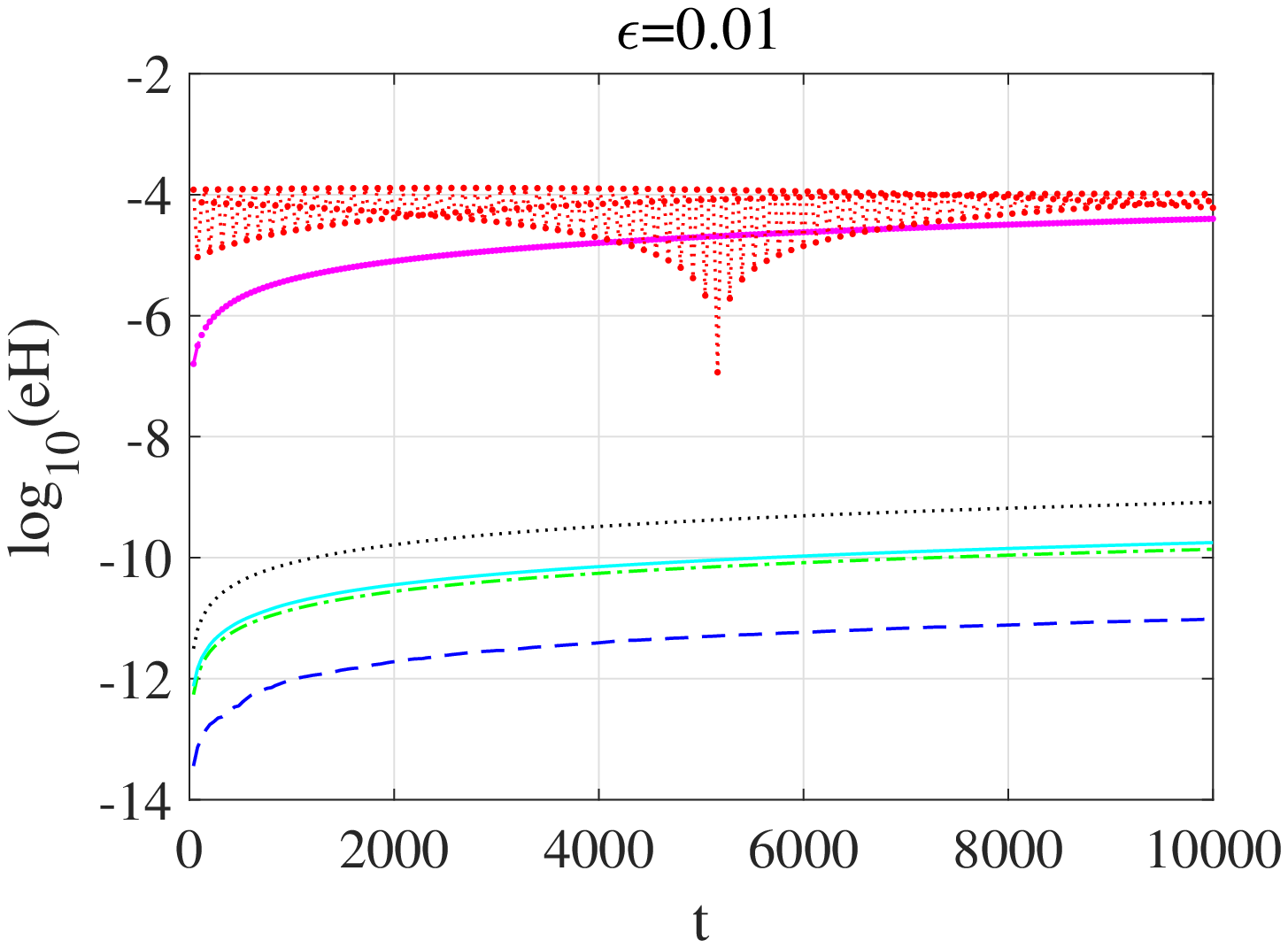}}		
		\end{tabular}
		\caption{{Problem \ref{prob4}.} Evolution of the error (\ref{ene-err}) with the modified energy $\tilde{H}(v,r)=\frac{1}{2}\norm{v}^2+r^2-1$ as function of time $t_n = nh$ under different $\eps$.}\label{fig41}
	\end{figure}
	
	\begin{figure}[t!]
		\centering
		\begin{tabular}[c]{ccc}%
			\subfigure{\includegraphics[width=4.7cm,height=4.2cm]{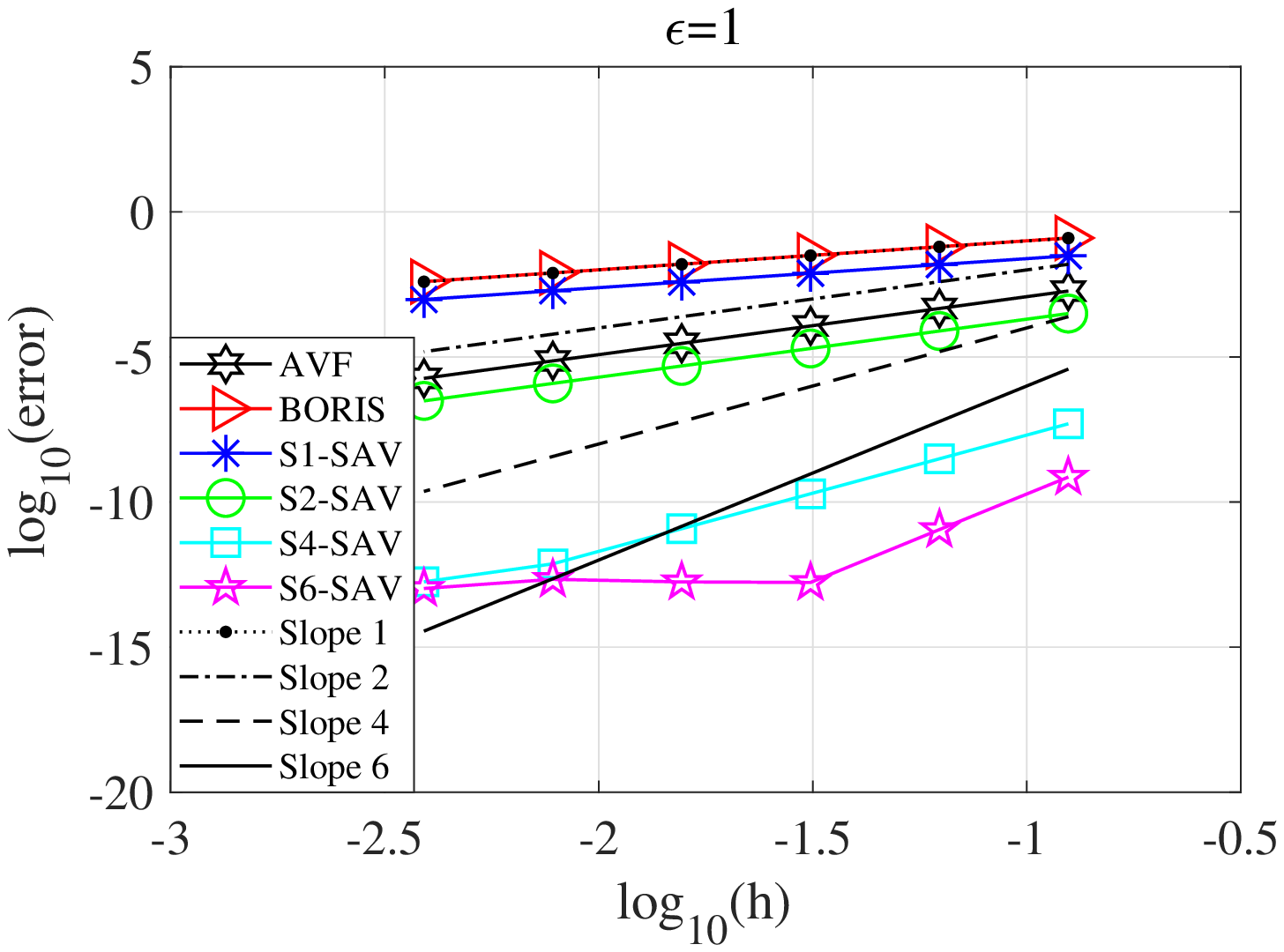}}			\subfigure{\includegraphics[width=4.7cm,height=4.2cm]{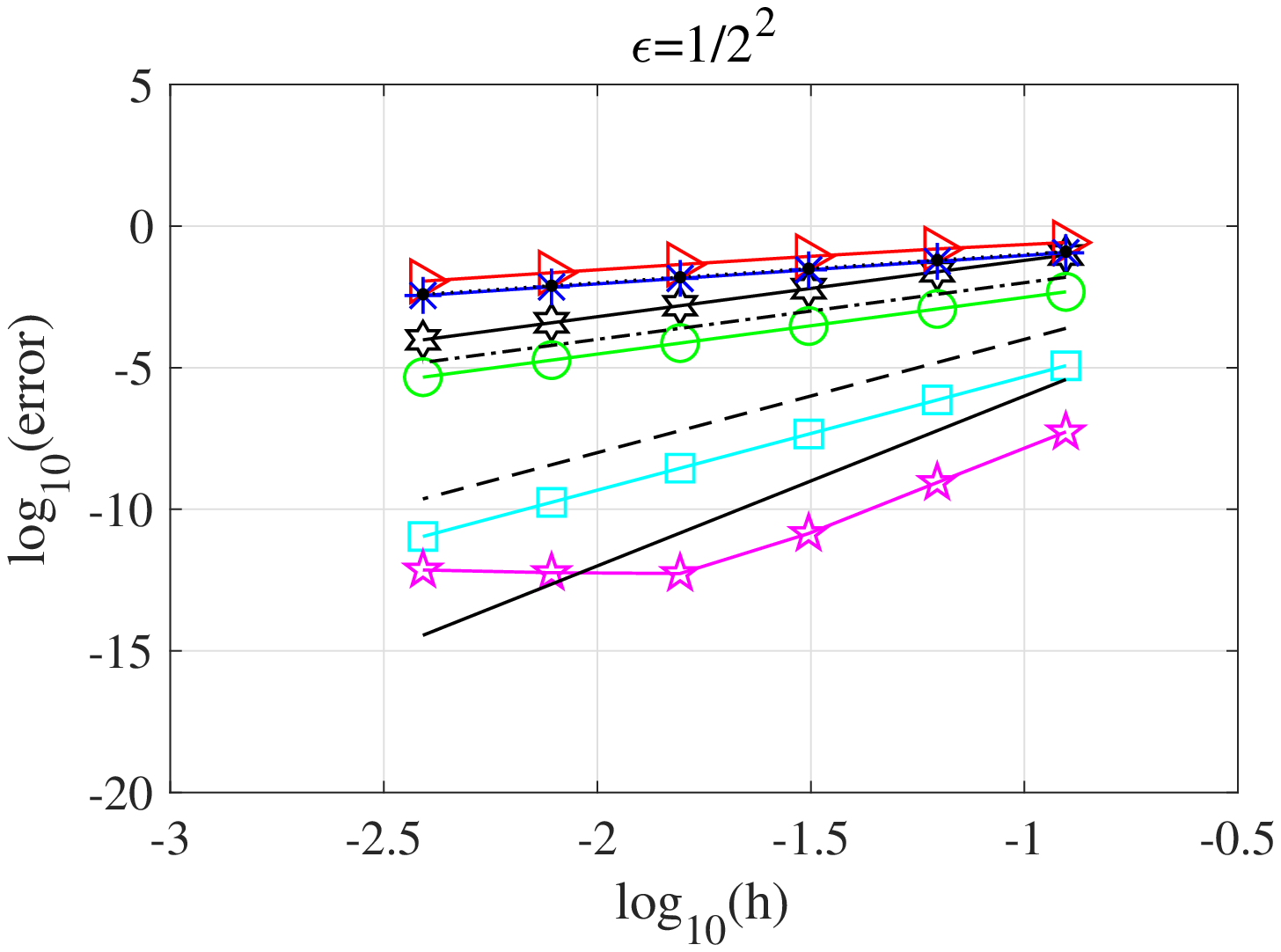}}
			\subfigure{\includegraphics[width=4.7cm,height=4.2cm]{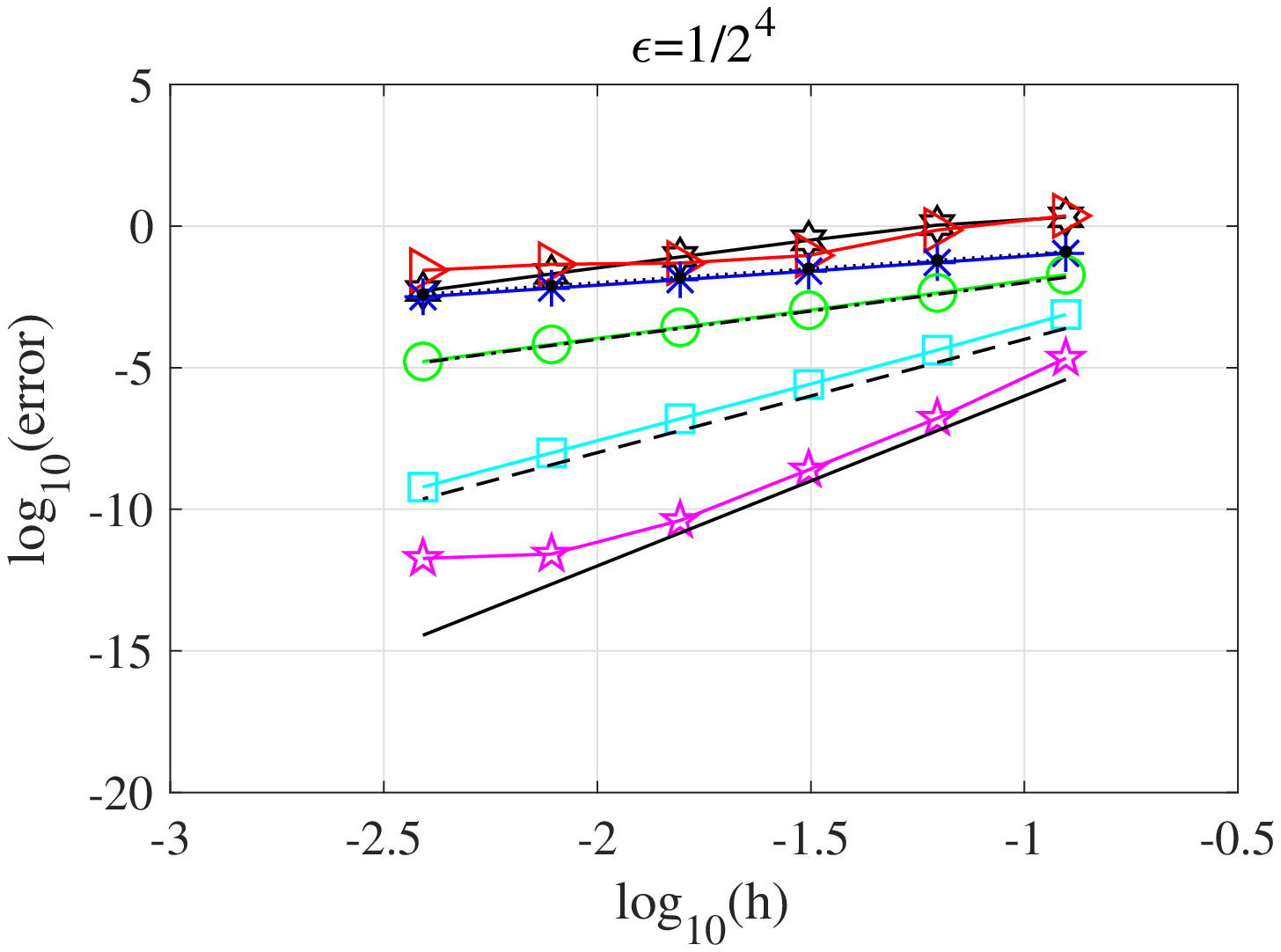}}		
		\end{tabular}
		\caption{{Problem \ref{prob4}.} The global errors \eqref{global error} with $T = 1$ and $h = 1/2^k$ for $k = 3,\dots,8$ under different $\eps$.}\label{fig42}
	\end{figure}
	
	\begin{figure}[t!]
		\centering
		\begin{tabular}[c]{ccc}%
			\subfigure{\includegraphics[width=4.7cm,height=4.2cm]{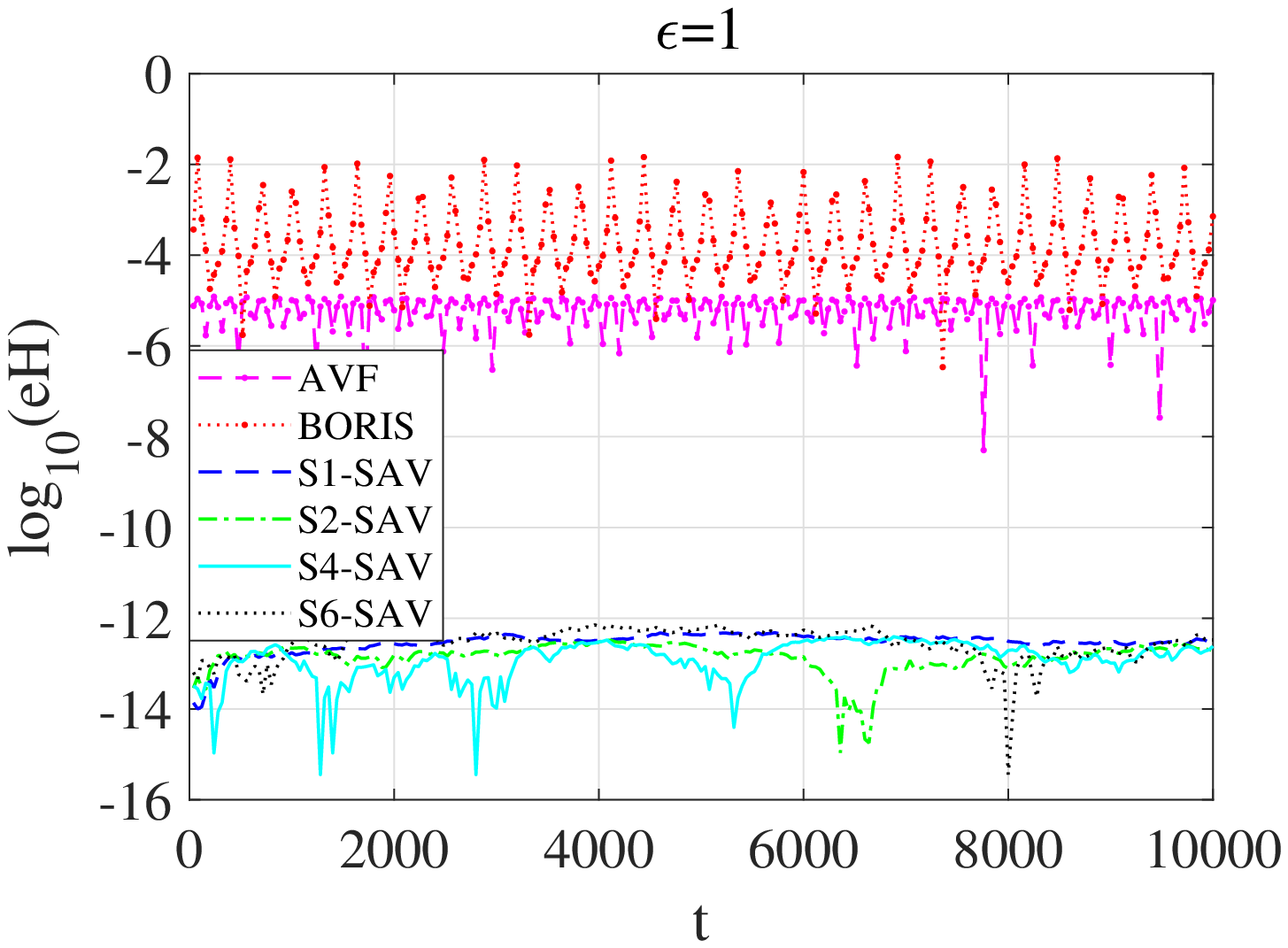}}			\subfigure{\includegraphics[width=4.7cm,height=4.2cm]{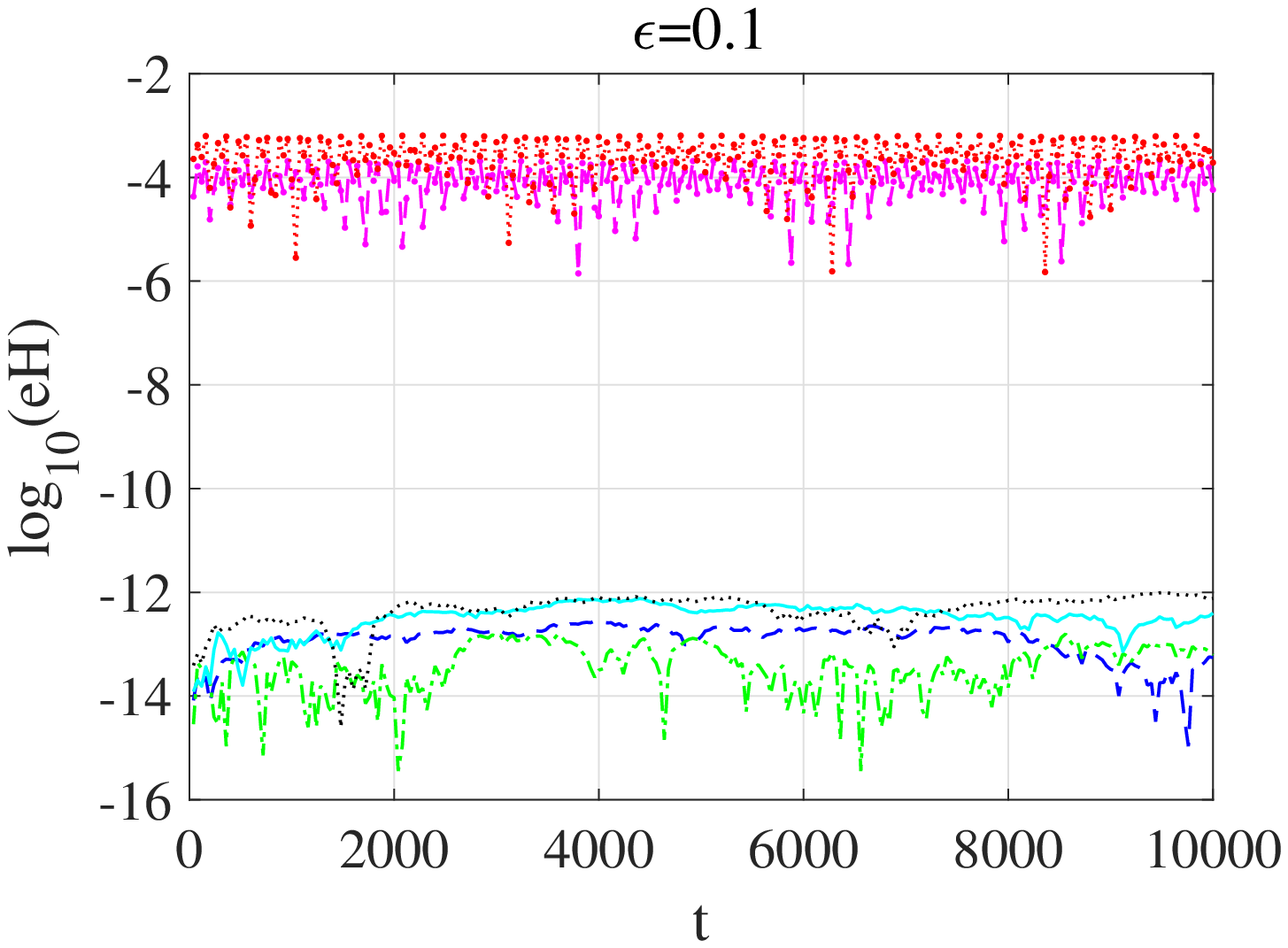}}
			\subfigure{\includegraphics[width=4.7cm,height=4.2cm]{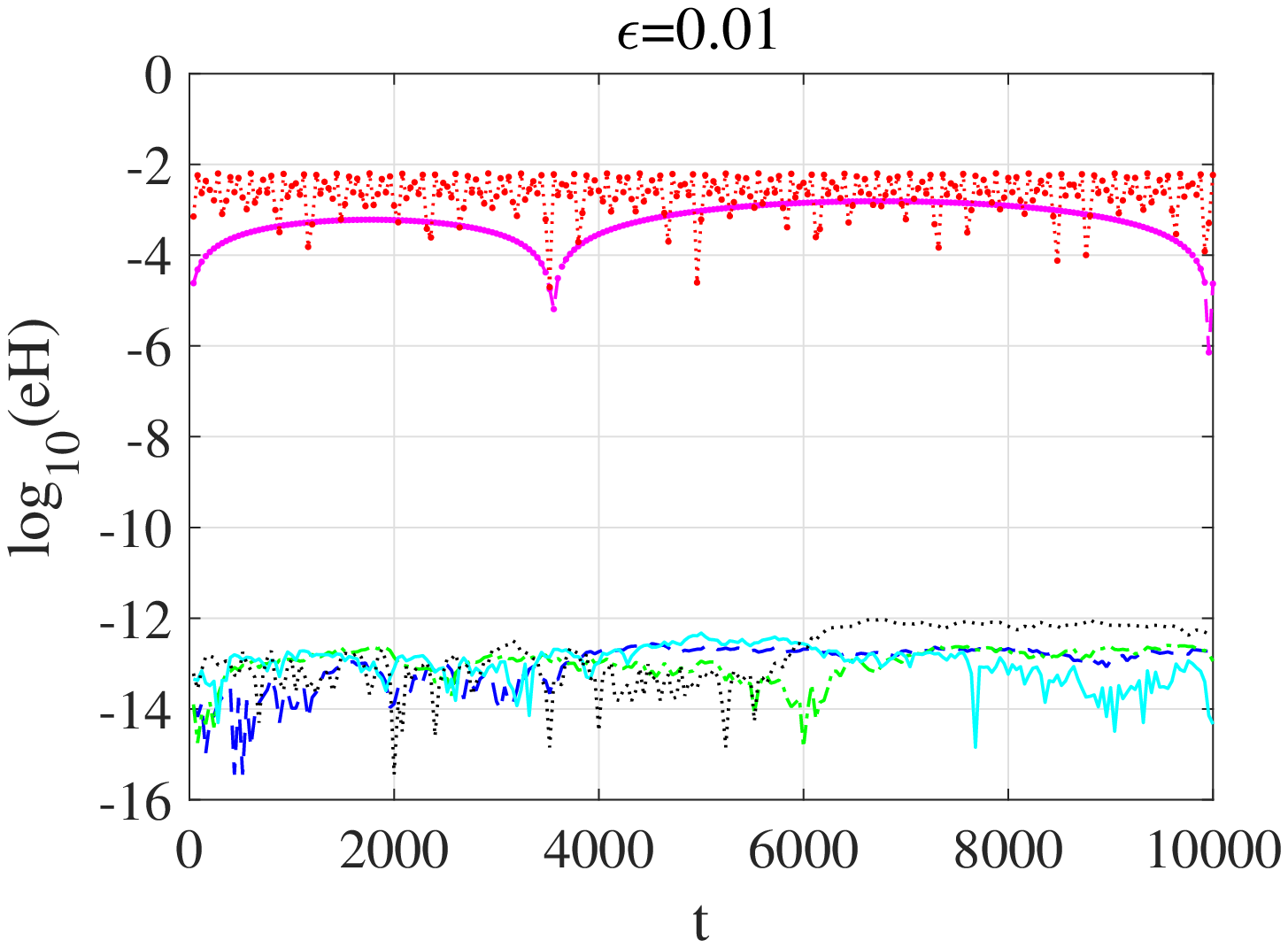}}		
		\end{tabular}
		\caption{{Problem \ref{prob5}.} Evolution of the error (\ref{ene-err}) with the modified energy $\tilde{H}(v,r)=\frac{1}{2}\norm{v}^2+r^2-1$ as function of time $t_n = nh$ under different $\eps$.}\label{fig51}
	\end{figure}
	
	\begin{figure}[t!]
		\centering
		\begin{tabular}[c]{ccc}%
			\subfigure{\includegraphics[width=4.7cm,height=4.2cm]{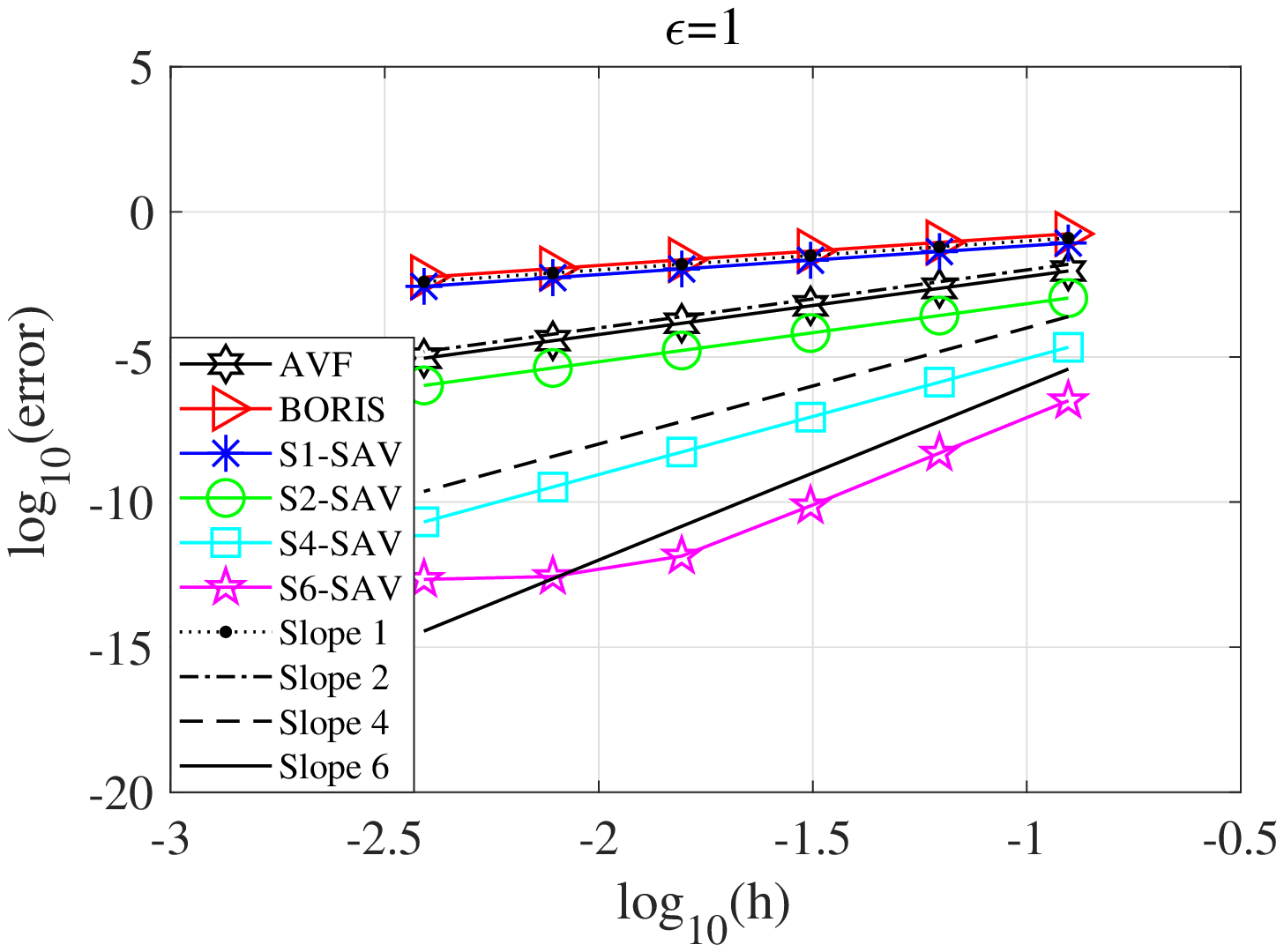}}			\subfigure{\includegraphics[width=4.7cm,height=4.2cm]{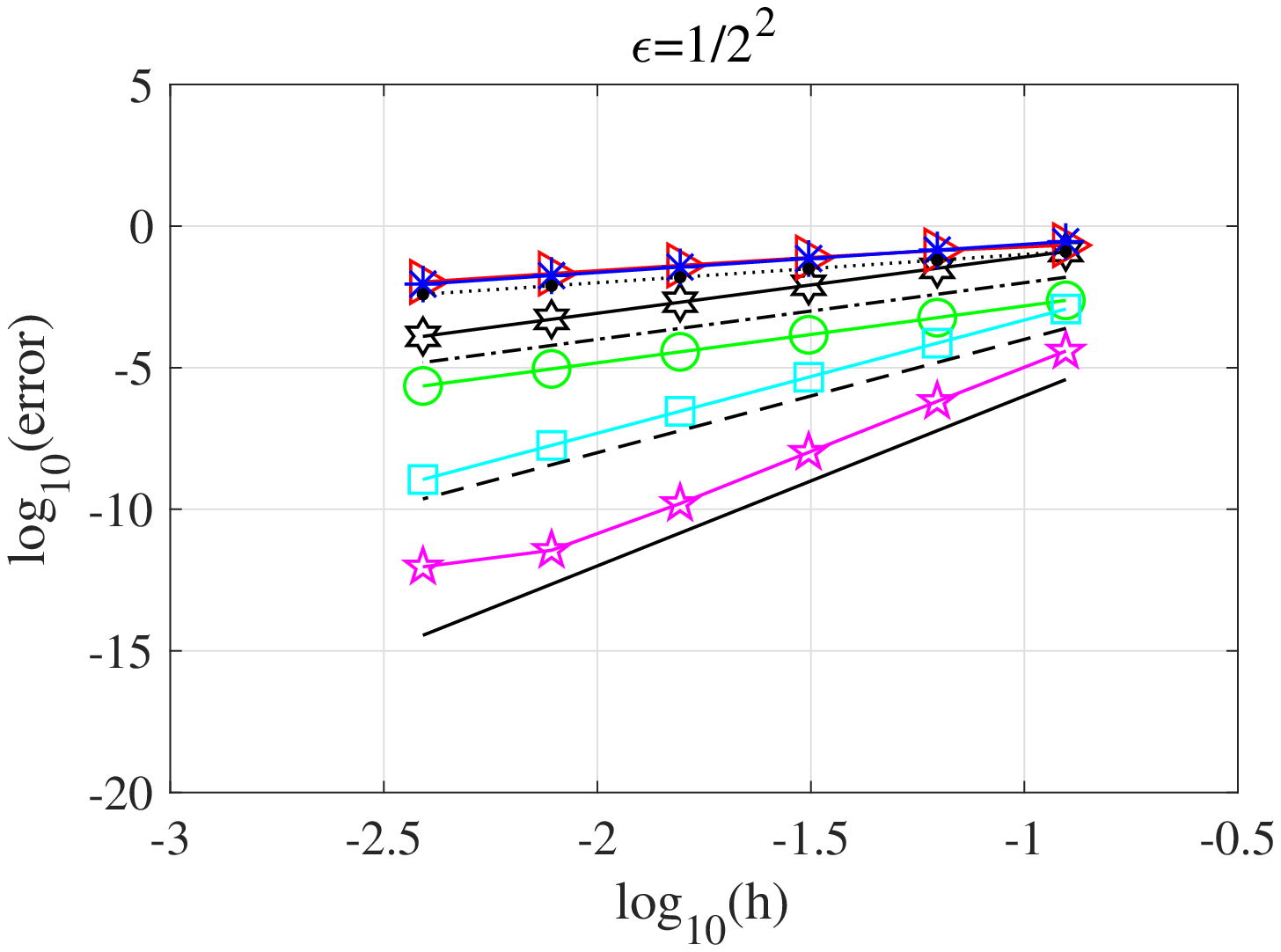}}
			\subfigure{\includegraphics[width=4.7cm,height=4.2cm]{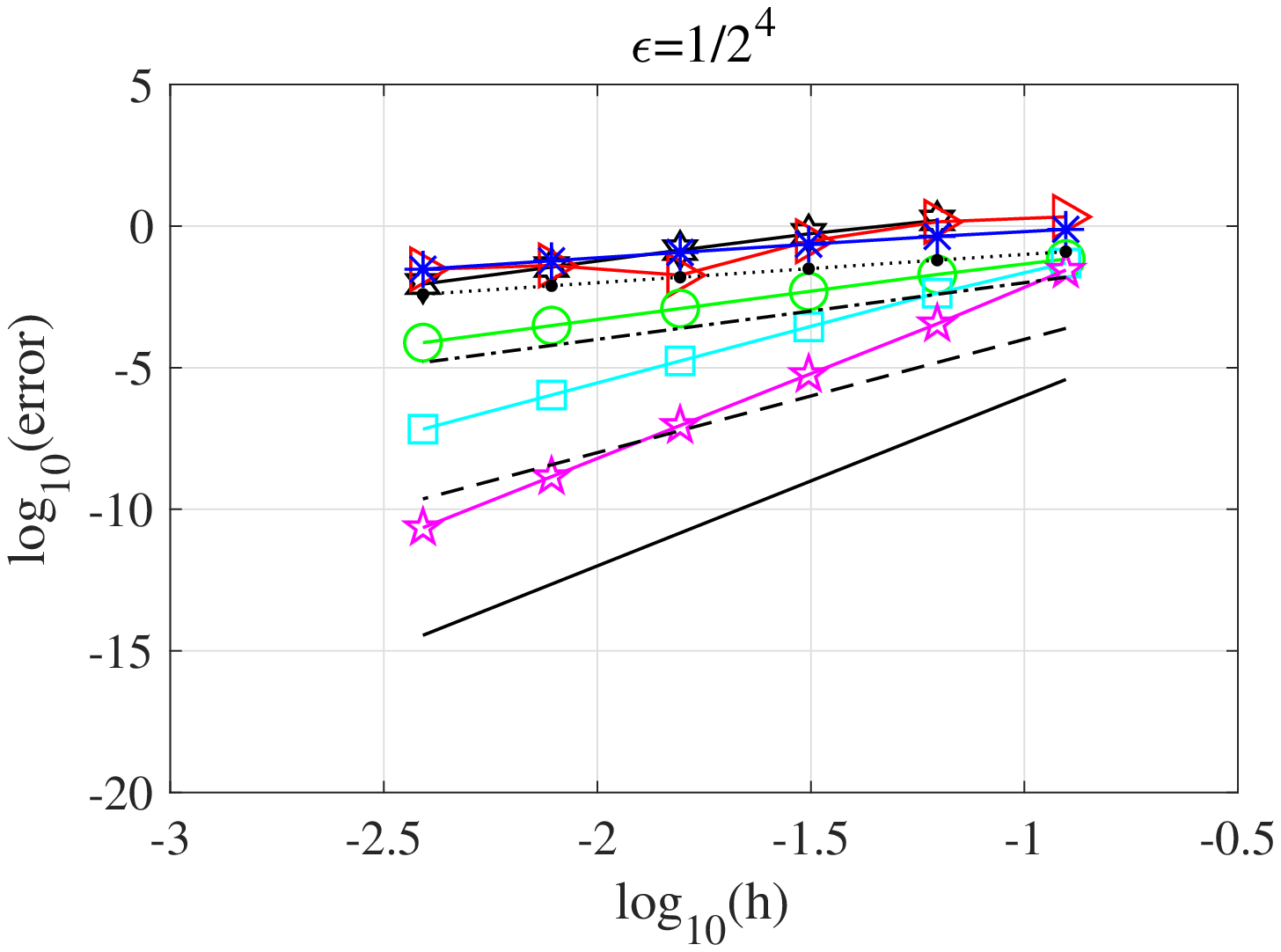}}		
		\end{tabular}
		\caption{{Problem \ref{prob5}.} The global errors \eqref{global error} with $T = 1$ and $h = 1/2^k$ for $k = 3, ..., 8$ under different $\eps$.}\label{fig52}
	\end{figure}
	
	\begin{figure}[t!]
		\centering
		\begin{tabular}[c]{ccc}%
			\subfigure{\includegraphics[width=4.7cm,height=4.2cm]{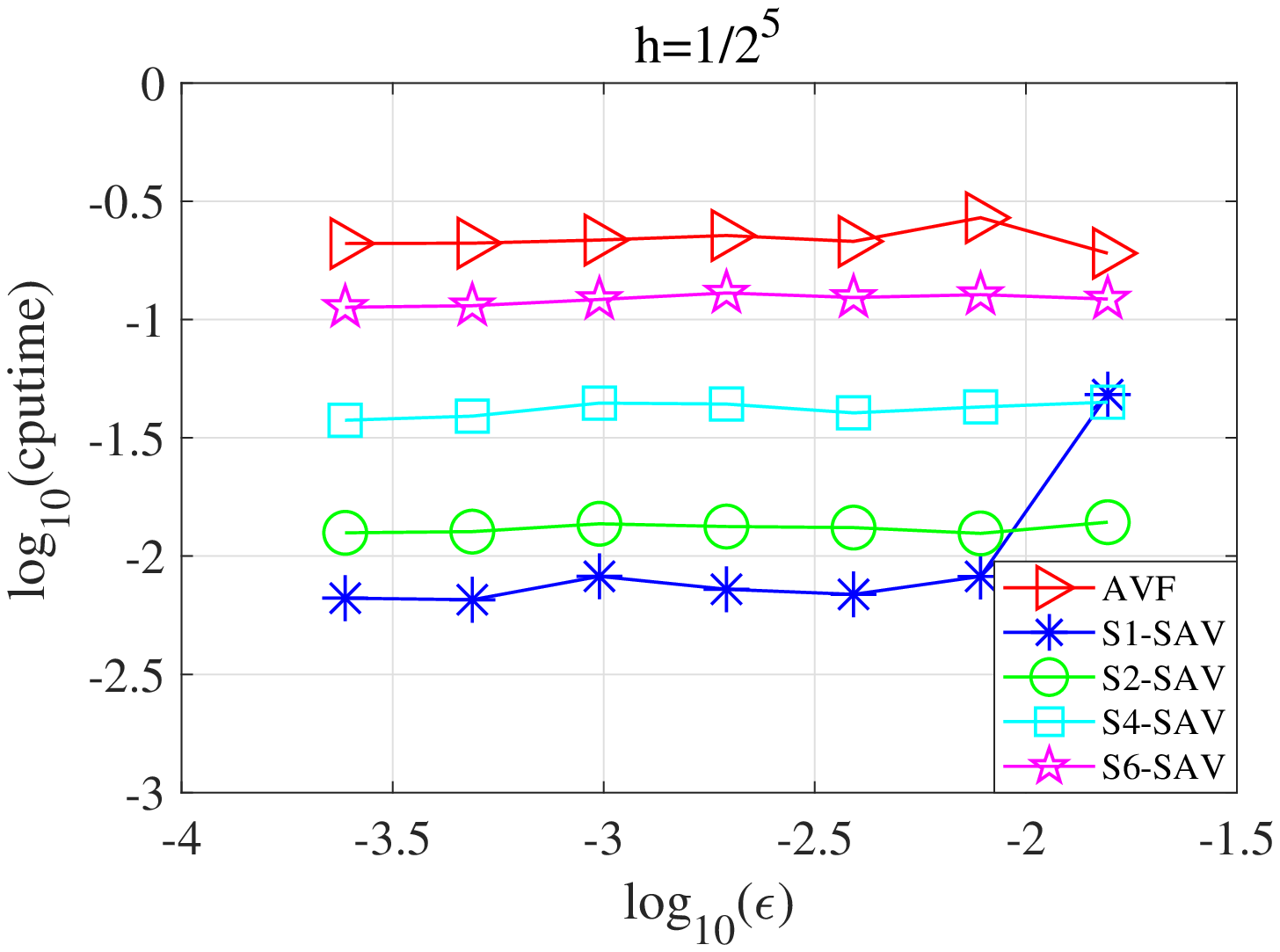}}			\subfigure{\includegraphics[width=4.7cm,height=4.2cm]{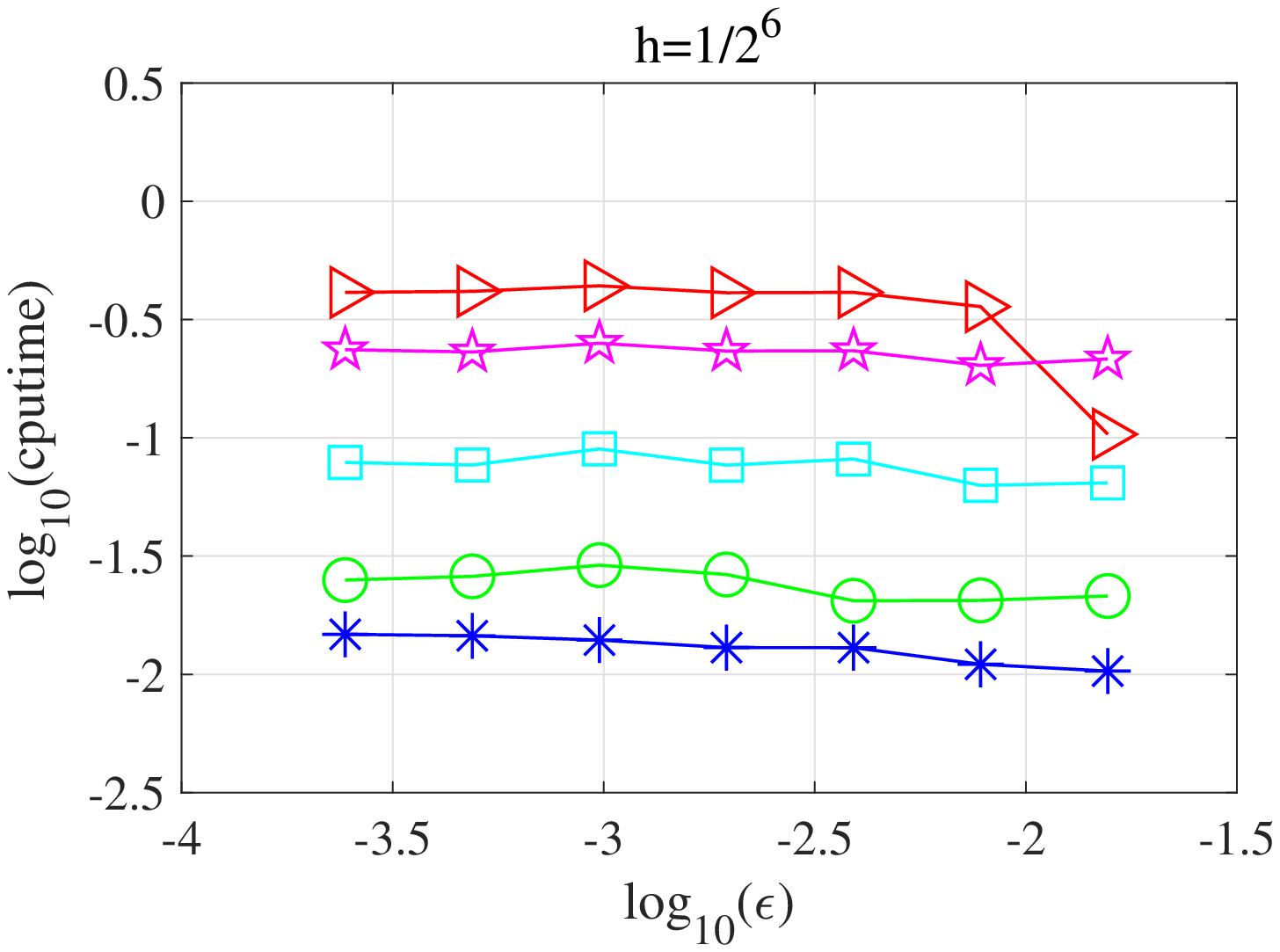}}
			\subfigure{\includegraphics[width=4.7cm,height=4.2cm]{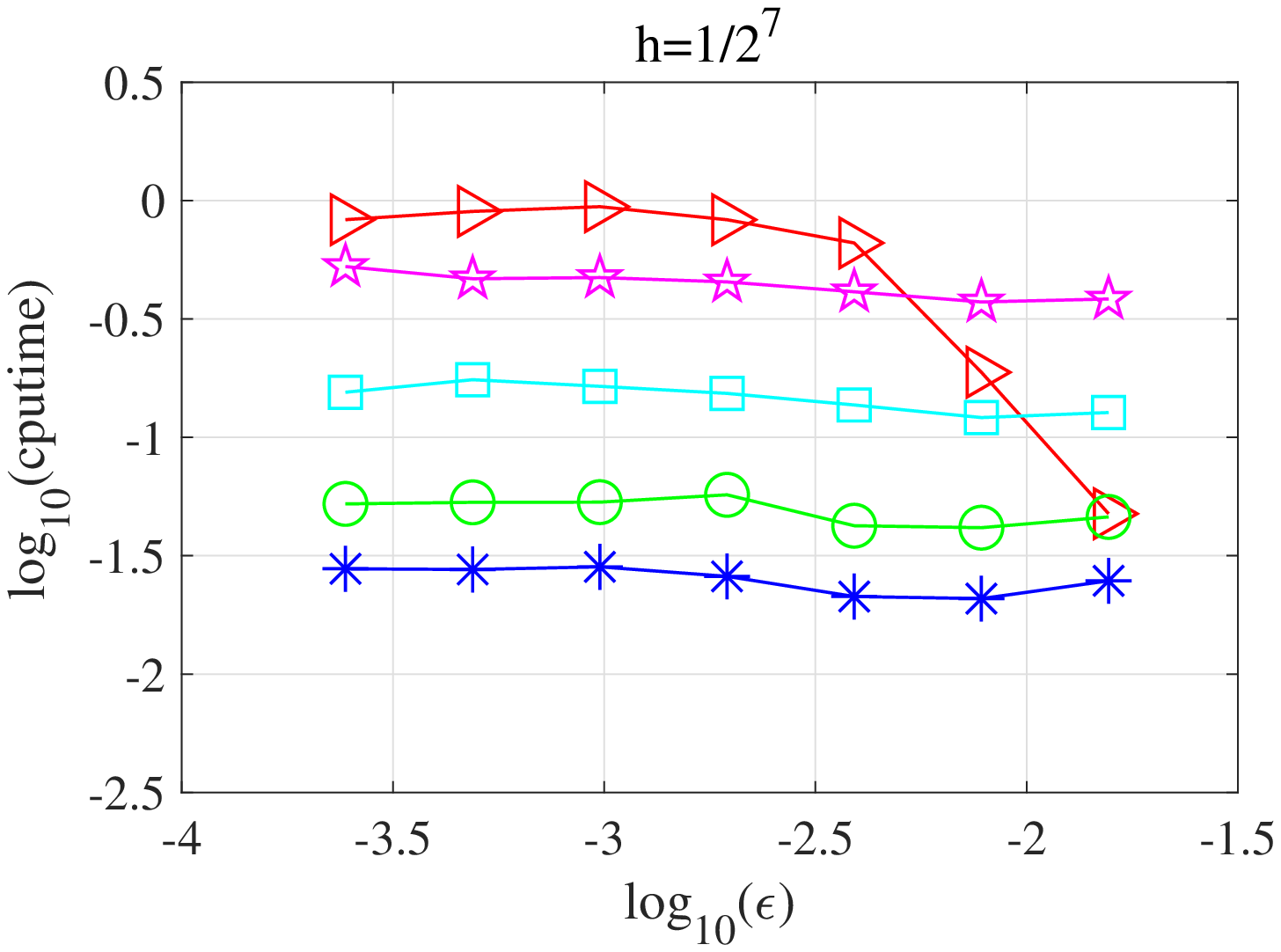}}		
		\end{tabular}
		\caption{{Problem \ref{prob5}.} Cputime of AVF and SSAVs under different $\eps=1/2^k$ for $k = 6,\dots,12$.}\label{fig53}
	\end{figure}
	
	\section{Conclusion}
	In this paper, we proposed and studied a new class of linearly implicit schemes for solving two  conservative systems: the oscillatory second-order differential equations (OSDE) and the charged-particle dynamics (CPD).  For the OSDE, by means of SAV approach and exponential integrators, we constructed a linearly implicit energy-preserving scheme (E2-SAV) of second order.  Combined with the splitting methods,   E2-SAV was successfully extended to the CPD to get a class of linearly implicit energy-preserving splitting schemes SSAVs. The rigorous analysis was given to show the excellent properties of the proposed methods including energy preservation and convergence. Moreover, we presented five numerical experiments, which highlighted the effectiveness of our schemes and confirmed the theoretical results.
	\section*{Acknowledgments}
	This work was supported by  Key Research and Development Projects of Shaanxi Province (2023-YBSF-399).
		
\end{document}